\documentclass[reqno, 10pt, centertags,draft]{amsart}
\usepackage{amsmath,amsthm,amscd,amssymb,latexsym,upref,stmaryrd}

\makeatletter
\numberwithin{equation}{section}

\newtheorem{theorem}{Theorem}[section]
\newtheorem{proposition}[theorem]{Proposition}
\newtheorem{corollary}[theorem]{Corollary}
\newtheorem{lemma}[theorem]{Lemma}
\newtheorem{remark}[theorem]{Remark}
\newtheorem{definition}[theorem]{Definition}
\newtheorem{example}[theorem]{Example}

\def\wt{\widetilde}

\def\cH{\mathcal H}

\def \gH{\mathfrak H}

\def \gl{\lambda}

\def \Ker{\text{Ker}}

\newcommand{\col}{\mathrm{col}}
\DeclareMathOperator{\diag}{diag}
\DeclareMathOperator{\Span}{span}
\DeclareMathOperator{\supp}{supp}
\DeclareMathOperator{\const}{const}
\DeclareMathOperator{\dom}{dom}
 \DeclareMathOperator{\tr}{tr}
\DeclareMathOperator{\im}{Im}

\renewcommand{\phi}{\varphi}

\begin{document}

\sloppy

\title[On the completeness of the system of root vectors]
{On the completeness of the system of root vectors for first-order
systems}

\author{M.M. Malamud, L.L. Oridoroga}

   \begin{abstract}
The paper is concerned with the completeness problem of root
functions of general boundary value problems for first order
systems of ordinary differential equations. Namely, we introduce
and investigate the class of \emph{weakly regular boundary
conditions}. We show that this class is much broader than the
class of {\em regular boundary conditions} introduced by
G.D.~Birkhoff and R.E.~Langer. Our main result states that the
system of root functions of a boundary value problem is complete
and minimal provided that the boundary conditions are weakly
regular. Moreover, we show that in some cases \emph{the weak
regularity} of boundary conditions \emph{is also necessary} for
the completeness. Also we investigate the completeness for
$2\times 2$ Dirac and Dirac type equations subject to irregular or
even to degenerate boundary conditions.

We emphasize that our results are the first results on the
completeness problem for general first order systems even in the
case of regular boundary conditions.
      \end{abstract}

\maketitle

\section{Introduction}

Spectral theory of non-selfadjoint boundary  value problems (BVP)
for $n$th order  ordinary differential equations (ODE)
   \begin{equation}\label{0.1}
y^{(n)} + q_1y^{(n-2)} + ... + q_{n-1}y = \lambda ^n y
   \end{equation}
on a finite interval $\mathcal{I}=(a,b)$ takes  its origin in the
classical papers by Birkhoff~\cite{Bir08},~\cite{Bir08exp} and
Tamarkin~\cite{Tam12}, ~\cite{Tam17},~\cite{Tam28}. They
introduced the concept of {\em regular boundary conditions} (BC)
and investigated the asymptotic behavior of eigenvalues and
eigenfunctions of such problems for ODE. Moreover, they proved
that the system of root functions, i.e. eigenfunctions and
associated functions (EAF) of the regular BVP is complete. Their
results are also treated in classical monographs (see, for
instance, \cite[Section 2]{Nai69} and \cite[Chapter
19]{DunSch71}).

However, some natural and important boundary conditions are not
regular. For instance, a boundary value problem with separated
boundary conditions is regular if and only if $n=2l$, where $l$ is
the number of boundary conditions   at the left (right) endpoint
of the interval $\mathcal{I}$. Note that the completeness of EAF
of boundary value problems with an arbitrary separated BC was
stated (without proof) much later by M.V. Keldysh in his famous
communication \cite{Kel51}. However, the proof of this result was
first appeared in the paper by A.A. Shkalikov \cite{Shk76}. The
completeness property of other non-regular BVP for $n$th order
ordinary differential equations on $[0, 1]$ has been studied by
A.G. Kostyuchenko and A.A. Shkalikov \cite{KosShk78}, A.P. Khromov
\cite{Khr3}, V.S. Rykhlov and many others.

On the other hand,  V.P.~Mihailov~\cite{Mikh62} and
G.M.~Keselman~\cite{Kes64} independently proved that the system of
EAF of a boundary value problem for equation \eqref{0.1} forms a
Riesz basis provided that the boundary conditions are {\em
strictly regular}. Similar results are also obtained
in~\cite[Chapter 19.4]{DunSch71}. Moreover, for boundary
conditions which are regular but not strictly regular,
A.A.~Shkalikov \cite{Shk79},~\cite{Shk82} proved that the system
of EAF forms a Riesz basis of subspaces.

In this paper we consider first order systems of ODE of the form
      \begin{equation} \label{1.1}
Ly :=  L(Q)y := \frac{1}{i}B \frac{dy}{dx} + Q(x)y = {\lambda}
y,\quad y = \col(y_1,...,y_n),
      \end{equation}
where $B$ is a non-singular diagonal $n\times n$ matrix,
     \begin{equation} \label{1.2}
B = \diag(b^{-1}_1I_{n_1}, \ldots, b^{-1}_rI_{n_r}) \in {\Bbb
C}^{n\times n}, \qquad n=n_1+...+n_r,
      \end{equation}
with complex entries  satisfying $b_j\ne b_k$ for $j\ne k$, and
$Q(\cdot)$ is a potential matrix. We also assume that $Q(\cdot)\in L^2([0,1];
{\Bbb C}^{n\times n})$. In the sequel we consider its
block-matrix representation $Q=(q_{jk})^r_{j,k=1}$ with respect to
the orthogonal decomposition ${\Bbb C}^n={\Bbb C}^{n_1}\oplus
...\oplus {\Bbb C}^{n_r}$. With
the system \eqref{1.1} one  associates, in a natural way, the maximal operator $L = L(Q)$ acting in
$L^2([0,1];{\Bbb C}^n)$ on the domain $\dom(L) = W^1_2([0,1];{\Bbb C}^n).$

Note that, systems form  a more general object than ordinary
differential equations. Namely, the $n$th-order differential
equation \eqref{0.1} can be reduced to the system~\eqref{1.1} with
$r=n$ and ${b}_j = \exp\left(2\pi ij/n\right)$ (see~\cite{Mal99}).
The systems \eqref{1.1}  are of significant interest in some
theoretical and practical questions. For instance, if $n=2m$,
$B=\diag(I_m, -I_m)$ and $q_{11}= q_{22}=0$, the
system~\eqref{1.1} is equivalent to the Dirac
system~\cite{LevSar88}, \cite{Mar77}.
  Note also that equation~\eqref{1.1} is used
to integrate  the problem of $N$ waves arising in the nonlinear
optics~\cite{ZMNovPit80}.

To obtain a BVP, we adjoin to equation \eqref{1.1} the following boundary
conditions
   \begin{equation}\label{1.3}
Cy(0) +  D y(1)=0,  \qquad    C= (c_{jk}),\ \  D = (d_{jk}) \in
{\Bbb C}^{n\times n}.
   \end{equation}
We denote by $L_{C, D} := L_{C, D}(Q)$ the operator associated  in $L^2([0,1];
{\Bbb C}^n)$ with the  BVP~\eqref{1.1}\ --\
\eqref{1.3}. It is defined as the  restriction of $L = L(Q)$ to
the domain
   \begin{equation}\label{1.3BB}
\dom(L_{C, D}) = \{y\in W^1_2([0,1];{\Bbb C}^n):\ Cy(0) + D
y(1)=0\}.
   \end{equation}
Moreover, in what follows we always impose the maximality
condition
   \begin{equation}\label{1.4}
{\rm rank}(C\,\,D)=n,
  \end{equation}
or equivalently
\[
 \ker(CC^*+DD^*)=\{0\}.
\]

Apparently,  the spectral problem \eqref{1.1}--\eqref{1.3} has first
been investigated by G.~D.~Birkhoff and R.~E.~Langer
\cite{BirLan23}. Namely, they have extended some previous results
of~Birkhoff and Tamarkin on non-selfadjoint BVP for ODE  to the
case of BVP \eqref{1.1}--\eqref{1.3}.  More precisely,  they
introduced the concepts of \emph{regular and strictly regular boundary
conditions} \eqref{1.3} and investigated  the asymptotic behavior
of eigenvalues and  eigenfunctions of the corresponding BVP
(the operator $L_{C, D}$). Moreover, they proved  \emph{a pointwise
convergence result} on spectral decompositions of the operator
$L_{C, D}$ corresponding to the BVP \eqref{1.1}--\eqref{1.3}.

However, to the best of our knowledge  the problem  of the
completeness of the root system \emph{of a general  BVP}
\eqref{1.1}--\eqref{1.3}  has not been investigated yet. Some
results in this direction were known only for the case of Dirac
systems. The present paper presents \emph{the first results in
this direction.}
More precisely, we introduce the concept of \emph{weakly regular} BC
for the  system \eqref{1.1} and establish the completeness of
EAF for this class of BVP (note that this class contains boundary conditions which are
regular in the sense of \cite{BirLan23}).

To state the main results, we need to the following construction.
Let $A=\diag(a_1,\dots,a_n)$ be a diagonal matrix with entries
$a_k$ (not necessarily distinct) that are not lying on the
imaginary axis, $\Re a_k \ne 0$. Starting with arbitrary matrices
$C, D\in \mathbb{C}^{n\times n}$, we define the auxiliary matrix
$T_A(C,D)\in  \mathbb{C}^{n\times n}$ as follows:
\begin{itemize}
\item if $\Re
a_k>0$, then the $k$th column in the
matrix $T_A(C,D)$ coincides with the $k$th column of the matrix $C$,
\item if $\Re
a_k<0$, then the $k$th column in the
matrix $T_A(C,D)$ coincides with the $k$th column of the matrix $D$.
\end{itemize}
 It is clear that $T_A(C,D)=T_{-A}(D,C)$.

Let us recall the definition of regular boundary conditions from
\cite{BirLan23}. Consider the lines $l_j := \{\lambda \in
\mathbb{C} : \Re(i b_j \lambda)=0\}$, $j \in \{1,2,\ldots,r\}$, of
the complex plane. The lines  $l_j$ divide the complex plane in $m
\le 2r$ sectors  $\sigma_1, \sigma_2, \ldots \sigma_m$. Let $z_1,
z_2, \ldots z_m$ be  complex numbers such that $i z_j$ lies in the
interior of $\sigma_j, j\in \{1,\ldots, m\}$. The boundary
conditions \eqref{1.3} are called regular whenever
\begin{equation}\label{regular}
    \det T_{z_jB}(C,D) \ne 0 , \qquad  j\in \{1,\ldots, m\}.
\end{equation}
Note that the boundary conditions \eqref{1.3} are regular if and only
if $\det T_{zB}(C,D) \ne 0$ for every  \emph{admissible} $z \in
\mathbb{C},$ i.e. for such $z$ that $\Re(zB)$ is non-singular.
  \begin{definition}\label{def1.1}
The boundary conditions  \eqref{1.3} are called \emph{weakly
$B$-regular} (or, simply, weakly regular) if there exist three
complex numbers $z_1, z_2, z_3,$ satisfying the following
conditions:

(a) the origin is an interior point of the triangle $\triangle
_{z_1z_2z_3};$

(b) $\det \,T_{z_jB}(C,D)\not = 0   \quad\text{for}\quad   j\in
\{1,2,3\}$.
\end{definition}

Now the first  main result of the paper reads as follows.
  \begin{theorem}\label{th2.1}
Let $Q\in  L^2([0,1]; {\Bbb C}^{n\times n})$ and let boundary
conditions  \eqref{1.3} be weakly $B$-regular. Then the system of
root functions of the BVP ~\eqref{1.1}--\eqref{1.3} (of the
operator $L_{C,D}(Q)$) is complete and minimal in $L^2[0,1]\otimes
\mathbb{C}^n$.
  \end{theorem}

We emphasize that the class of \emph{weakly regular boundary
conditions is much wider than  the class of regular BC}. For
instance, for splitting  boundary conditions  \eqref{1.3} to be
regular it is necessary that: (i) $n=2k$, where $k$ is the number
of conditions at zero; (ii) the matrix $\Re(zB)$ has zero
signature for every admissible $z.$ However,  for odd $n=2k+1$
splitting BC with  $k$ conditions at 0 are \emph{weakly
$B$-regular}, in general, whenever $b_j=\exp\left(\frac{2 \pi i
j}{n}\right)$  (see Example~\ref{ex.k.k+1} for details). Moreover,
there exist splitting \emph{irregular} but  \emph{weakly regular}
BC for $n=2k$ too.

In the case of $B=B^*$ weak regularity of boundary conditions
\eqref{1.3} is equivalent to their  regularity. Moreover, denoting
by  $P_+$ and $P_-$ the spectral projectors onto "positive"\ and
"negative"\ parts of the spectrum of $B=B^*$, respectively, one
expresses  the regularity of boundary conditions  \eqref{1.3}  as
follows:
   \begin{equation}\label{2.4Intro}
\det(CP_{+}  +\  DP_{-}) \ne 0 \quad\text{and}  \quad \det(CP_{-}
+\ DP_{+}) \ne 0.
\end{equation}

Thus, Theorem \ref{th2.1} yields the following result.
    \begin{corollary}\label{Cor1.3Intro}
Let $Q\in L^2[0,1]\otimes {\Bbb C}^{n\times n},$   $B=B^*$ and let
conditions \eqref{2.4Intro} be satisfied. Then the system of root
functions of the BVP ~\eqref{1.1}--\eqref{1.3}  is complete and
minimal in $L^2[0,1]\otimes \mathbb{C}^n$.
   \end{corollary}

In some particular  cases  this statement has  been obtained  by
V.A. Marchenko \cite{Mar77}  $(2\times 2 \ \text{Dirac system,}\ B
= \diag(-1,1))$ and V.P. Ginzburg \cite{Gin71}  $(B=I_n, Q=0)$
(see Remark \ref{rem3.3} below).

 Note that   conditions \eqref{2.4Intro} are also necessary for
completeness if  $Q=0.$ However, they  \emph{are no longer
necessary} if  $Q\not \equiv 0$ even for $Q = Q^*.$ We
demonstrate this fact in passing by  stating  a special case of
Theorem  \ref{th3.1}  that gives   \emph{new  conditions of the
completeness of irregular} BVP for $2\times 2$ Dirac systems.
   \begin{proposition}\label{prop1.2_Intro}
Let $B=\diag(-1,1)$, $Q=
\begin{pmatrix}
0&Q_{12}\\
Q_{21}&0
\end{pmatrix}$ and  $Q_{12}(\cdot),
Q_{21}(\cdot) \in C[0,1].$ Assume that
\begin{equation}  \label{3.7AAIntro}
J_{13}Q_{12}(0) - J_{42}Q_{21}(1) \ne 0,  \qquad
J_{13}Q_{12}(1) - J_{42}Q_{21}(0) \ne 0,
\end{equation}
where $J_{13}:=\det
\begin{pmatrix}
c_{11}&d_{11}\\
c_{21}&d_{21}
\end{pmatrix}, J_{42}:=\det
\begin{pmatrix}
d_{12}&c_{12}\\
d_{22}&c_{22}
\end{pmatrix}$.
Then the system of root functions of the
problem~\eqref{1.1}--\eqref{1.3} is complete and minimal in
$L^2\left([0,1];{\Bbb C}^2\right)$.
   \end{proposition}
We emphasize that the assumptions  of Proposition
\ref{prop1.2_Intro} depend on $Q$ although they guarantee the
completeness even if both conditions \eqref{2.4Intro} are
violated. However, these assumptions cover  \emph{irregular} and
even \emph{degenerate BC}~\eqref{1.3}.

In connection with Corollary \ref{Cor1.3Intro} and Proposition
\ref{prop1.2_Intro} we mention the papers
~\cite{TroYam01},~\cite{TroYam02}, \cite{HasOri09}, \cite{Mit04}
and~\cite{DjaMit09Sing}, \cite{DjaMit09}, \cite{DjaMit09Hill},
\cite{DjaMit10Trig}, \cite{DjaMit10}, \cite{DjaMit11Crit},
\cite{DjaMit11Equi}, \cite{DjaMit11CritDir}, that appeared during
the last decade. Basically they are devoted to the Riesz basis
property of EAF for BVP with  strictly regular (and just regular)
BC for $2\times 2$ Dirac systems. The most complete and detailed
results in this direction have been obtained by P.~Djakov and
B.~Mityagin~\cite{DjaMit09Sing},~\cite{DjaMit09},~\cite{DjaMit10},
\cite{DjaMit11Crit}, \cite{DjaMit11Equi}, \cite{DjaMit11CritDir}.
 In the recent preprint~\cite{DjaMit11Equi} they proved
equiconvergence and pointwise convergence of spectral
decompositions of  Dirac operators with regular BC. The result on
pointwise convergence  improves and generalizes the corresponding
result from~\cite{BirLan23} for $2\times2$ Dirac systems.
Moreover,  in \cite{DjaMit11Crit}, \cite{DjaMit11CritDir}  a
\emph{criterion} for EAF to form a Riesz basis for periodic
(resp., antiperiodic) 1D Dirac operator is established .

Let us also  mention the recent papers by F. Gesztesy and V.
Tkachenko~\cite{GesTka09},~\cite{GesTka11}. In particular,
in~\cite{GesTka11}, as well as in the recent preprint by P.~Djakov
and B.~Mityagin~\cite{DjaMit11Crit},  the authors established a
\emph{criterion}  for eigenfunctions and associated functions to
form a Riesz basis for periodic (resp., antiperiodic)
Sturm--Liouville operators on $[0,1]$. This criterion is
formulated  in terms of periodic (resp., antiperiodic) and
Dirichlet eigenvalues.

Note also that using the approach from \cite{MasVor70} Theorems
\ref{th2.1} and \ref{th3.1}  can be applied for the study  of
uniqueness of mixed BVP for first order systems of partial
differential equations.

The paper is organized as follows. In Section 2 we present a
result on asymptotic behavior of solutions of equation \eqref{1.1}
as $\lambda\to\infty$. This result generalizes the classical
Birkhoff result \cite{Bir08} (see also \cite{Nai69}) and completes
the result from \cite{BirLan23}.

In Section 3 we present the  proof of  Theorem \ref{th2.1}. We
also prove here (see Corollary \ref{cor2.1A}) that if the BC are
weakly regular, then the system of root functions of the adjoint
operator $L^*_{C,D}$ is complete and minimal too. Besides, we
present here some examples of irregular BC that are weakly
regular. In particular, we show that under very weak assumptions
the splitting BC are weakly regular.

In Section 4 we investigate the problem  \eqref{1.1}--\eqref{1.3}
with $B=B^*$ (Dirac type systems). We prove Corollary
\ref{Cor1.3Intro}. We also show that for dissipative
(accumulative) operators $L_{C,D}$ the first (the second)
condition in  \eqref{2.4Intro} yields completeness (see Corollary
\ref{cor2.2}). It is also proved here that in the case $Q=0$
conditions \eqref{2.4Intro} of (weak) regularity are necessary for
completeness.

In Section 5 we investigate boundary value problems  for $2\times
2$ Dirac type systems $(B=B^*)$ and present other sufficient
conditions of the completeness in the irregular case. In the proof
of the main result of the section, Theorem \ref{th3.1}, we
substantially exploit triangular transformation operators that
were constructed for general $n\times n$ Dirac type systems in
\cite{Mal99}.   For Dirac system we also find some necessary
conditions for completeness that show, in particular, the
sharpness of conditions \eqref{3.7AAIntro}  for the validity of
Proposition \ref{prop1.2_Intro} (see Proposition \ref{prop5.12}).

Finally, in Section 6 we investigate BVP \eqref{1.1}--\eqref{1.3}
for $n=2$ with $B = \diag(b^{-1}_1,b^{-1}_2)\not =B^*$    
and complete Theorem  \ref{th2.1} for this case.
Namely, in Theorem \ref{th7.1} we prove completeness and
minimality of the root functions of the BVP
\eqref{1.1},~\eqref{1.3} with  $C=
\begin{pmatrix}
1&-h_0\\
0&-h_1
\end{pmatrix}\
\text{and}\  D=\begin{pmatrix} 0&0\\
1&0
\end{pmatrix},\  h_0 h_1 \ne 0$,
when the BC \eqref{1.3} are not weakly regular. In this case
completeness of the adjoint operator $L_{C,D}(Q)^*$ depends on
$Q.$  However, we show in Corollary \ref{cor6.4}  that in the case
$B = \diag(b^{-1}_1,b^{-1}_2)\not =B^*$  \emph{weak $B$-regularity
of boundary conditions  \eqref{1.3} is equivalent to the
completeness of both operators $L_{C,D}(0)$  and $L_{C,D}(0)^*$
with $Q=0.$}

The main results of the paper have been announced  in
\cite{MalOri00, MalOri10}.

{\bf Notation.} We denote by $\left\langle
\cdot,\cdot\right\rangle$ the inner product in ${\Bbb C}^{n}.$
${\Bbb C}^{n\times n}$ stands for the set of $n\times n$ matrices
with complex entries; $I_n (\in {\Bbb C}^{n\times n})$ stands for
the unit matrix;  by $\kappa_+(A)$ ($\kappa_-(A)$) we denote  the
number of positive (negative) eigenvalues of the selfadjont matrix
$A$.

$o_n(1)$ stands for an $n\times n$ matrix function with entries of
the form $o(1)$;
 $[f(x)]$ stands for the function of the form $f(x)(1+o(1))$;

\section{Preliminaries}

\subsection{The asymptotic behavior
 of solutions to first-order systems}

Here we present a result  on the asymptotical growth of solutions
to first order systems of equations \eqref{1.1}. This result
slightly generalizes the corresponding result from
\cite[p.71-87]{BirLan23} on systems \eqref{1.1} where it was
obtained under a stronger  assumption $Q\in C^1[0,1]\otimes
C^{n\times n}.$ In turn, the latter result from \cite{BirLan23}
generalizes  the classical Birkhoff theorem on  $n$th-order
ordinary differential equation (see, for instance,~\cite{Bir08},
\cite{Nai69}). We present the proof for the sake of completeness.
Moreover, our exposition  slightly differs from that in
\cite{BirLan23} and is shorter.

To this end, we need the following lemma.
\begin{lemma}\label{LemBirk}
Let  $a_1, a_2,\dots, a_r$ are different complex numbers. Then the
complex plane can be divided into at most $r^2-r$ sectors $S_p$
with vertexes at the origin and  such that for any $p$ the numbers
$a_j$ can be renumbered so that  the following inequalities hold:
\begin{equation} \label{b2}
\Re (a_{j_1}\lambda) < \Re (a_{j_2}\lambda) <\dots <\Re
(a_{j_r}\lambda),  \qquad \lambda\in S_p.
\end{equation}
        \end{lemma}
\begin{proof}[Proof]
Let $l_{jk}$ be the set of $z$  satisfying $\Re(a_j z) = \Re
(a_kz)$. Then $l_{jk}$ is the line on the complex plane passing
through the origin. All such lines divide the complex plane into
at most $r^2-r$ sectors. Assume that $a_j$ are ordered in a such
way that inequalities~\eqref{b2} hold for a certain $\lambda_0$
lying inside a sector. In this case, since $\Re (a_{j_k}\lambda)
\ne \Re (a_{j_l}\lambda)$ for any $\lambda$ inside the sector and
all the functions $\Re (a_j\lambda)$, \ $j\in\{1,2,\dots,r\}$, are
continuous, it follows that the inequalities~\eqref{b2} are valid
for every $\lambda$ from the chosen  sector as well.
\end{proof}

Clearly, each of the sectors $S_p$ is of the form $S_p = \{
z:\varphi_{1p}< \arg z<\varphi_{2p}\}$. Fix $p$ and denote by $S$
the sector strictly embedded into the latter, i.e.,
\begin{equation} \label{bS}
\begin{split}
S&:=\{z: \varphi_{1p} + \varepsilon_1<\arg z<\varphi_{2p} -
\varepsilon_2\},
\quad  \text{where}\quad  \varepsilon_1, \varepsilon_2>0; \\
 S_R&:=\{z\in S: |z|>R\}.
  \end{split}
\end{equation}
%
%
     \begin{proposition}\label{BirkSys}
Assume that $B=\diag(b_1^{-1}I_{n_1}, b_2^{-1}I_{n_2}, \dots,
b_r^{-1}I_{n_r})$ is a nonsingular  diagonal $n\times n$ matrix
with  $b_j\ne b_k$ for $j\ne k$, and $Q(x)=(q_{jk}(x))_{j,k=1}^r$
where $q_{jk}(\cdot)\in L^1[0,1]\otimes \mathbb{C}^{n_j\times
n_k}$ and  $q_{jj}(\cdot)\equiv 0,\ \ j\in\{1,2,\dots,r\}$.
Further, let $S$ be the sector of the form \eqref{bS}.
Then the numbers $\{ib_j\}_1^r$ can be renumbered with respect to
the sector $S$ in accordance with~\eqref{b2}, i.e.
\begin{equation} \label{b2sect3}
\Re (ib_{j_1}\lambda) <\Re (ib_{j_2}\lambda) <\dots <\Re
(ib_{j_r}\lambda),\qquad \lambda\in S.
\end{equation}
Moreover, for a sufficiently large $R$, equation~\eqref{1.1} has
the fundamental system of  matrix solutions
   \begin{equation}  \label{b3}
Y_k(x;\lambda)=
 \begin{pmatrix} y_{1k}(x;\lambda) \\
 y_{2k}(x;\lambda) \\
 \dots \\
 y_{rk}(x;\lambda)
\end{pmatrix},
\qquad y_{jk}(\cdot;\lambda):[0,1]\to \mathbb{C}^{n_j\times n_k},
\quad k\in \{1,2,\dots,r\},
\end{equation}
which is analytic with respect to  $\lambda\in S_R$ and has the
asymptotic behavior (uniformly in $x$)
\begin{equation} \label{b4}
\begin{split}
y_{kk}(x;\lambda)&=(I_{n_k} + o(1)) e^{ib_k\lambda x},\qquad \lambda \in S_R, \\
y_{jk}(x;\lambda)&=o(1)e^{ib_k\lambda x},\qquad \lambda \in
S_R,\quad \text{for} \quad j\ne k.
\end{split}
\end{equation}
        \end{proposition}
\begin{proof}[Proof]
The first statement is immediate from Lemma \ref{LemBirk}.
 Without loss of generality
we assume that $b_{j_k}=b_k,\ k\in \{1,\ldots, r\}.$  Besides for
simplicity, we restrict ourselves to the case of the matrix  $B$
with simple spectrum, i.e., assume that $n_k=1$ for $k\in
\{1,\ldots, r\}$. In this case, $r=n$, and $Y_k(x;\lambda)$ is the
vector column with the components $y_{jk}(x;\lambda),\ \ j\in
\{1,2,\dots,n\}$.

Denote $\widetilde{q}_{jl}(t)=-i b_j q_{jl}(t)$.
It is easy to check that, for every fixed $k\in
\{1,2,\dots,r\}$, a solution of the system of
integral equations
\begin{equation}  \label{b5}
\begin{cases}
 \begin{split}
 y_{jk}(x;\lambda)=\int_0^x e^{ib_j\lambda(x-t)}
  \sum_{l=1}^r \widetilde{q}_{jl}(t) y_{lk}(t;\lambda)\, dt
 \end{split}&\text{ \ for } j<k,
 \\
 \begin{split}
 y_{jk}(x;\lambda)=e^{ib_j\lambda x}+
 \int_0^x e^{ib_j\lambda(x-t)}
  \sum_{l=1}^r \widetilde{q}_{jl}(t) y_{lk}(t;\lambda)\, dt
 \end{split}&\text{ \ for } j=k,
 \\
 \begin{split}
 y_{jk}(x;\lambda)=-\int_x^1 e^{ib_j\lambda(x-t)} \sum_{l=1}^r \widetilde{q}_{jl}(t)
 y_{lk}(t;\lambda)\,dt
 \end{split} &\text{ \ for } j>k,
\end{cases}
\end{equation}
is the solution to the system~\eqref{1.1} as well.

Let us verify that system~\eqref{b5} has a unique solution for
sufficiently large absolute values of $\lambda\in S$, and this
solution satisfies conditions~\eqref{b4}. Introduce new functions
$z_{jk}(x;\lambda)$ by setting
\begin{equation} \label{b6}
z_{jk}(x;\lambda):=e^{-ib_k\lambda x} y_{jk}(x;\lambda),\quad
j,k\in \{1,2,\dots,r\}.
   \end{equation}
Then the $k$-th equation in the system~\eqref{b5}
yields
\begin{equation}  \label{b7}
z_{kk}(x;\lambda)=1+\int_0^x\sum_{j=1}^r
\widetilde{q}_{kj}(t)z_{jk}(t;\lambda)\, dt.
\end{equation}
By substituting expressions~\eqref{b6} and~\eqref{b7} into the
system~\eqref{b5} we obtain
 \begin{equation} \label{b8}
 \begin{cases}
 \begin{split}
 z_{jk}(x;\lambda)&=\int_0^x \widetilde{q}_{jk}(t)e^{i(b_j - b_k)\lambda(x-t)}\,dt
 +{\sum_{1\le l\le r}}' \int_0^x e^{i(b_j- b_k)\lambda(x-t)} \times \\
 &\left(\times \widetilde{q}_{jl}(t) z_{lk}(t;\lambda) +
 \widetilde{q}_{jk}(t)\int_0^t \widetilde{q}_{kl}(\tau) z_{lk}(\tau;\lambda)\,d\tau\right)\,dt,
 \end{split}& j<k
 \\
 \begin{split}
 z_{jk}(x;\lambda)&=-\int_x^1 \widetilde{q}_{jk}(t)e^{i(b_j - b_k)\lambda(x-t)}\,dt-
 {\sum_{1\le l\le r}}' \int_x^1 e^{i(b_j- b_k)\lambda(x-t)}\times \\
 &\left(\times \widetilde{q}_{jl}(t)z_{lk}(t;\lambda) +
  \widetilde{q}_{jk}(t)\int_0^t \widetilde{q}_{kl}(\tau) z_{lk}(\tau;\lambda)\,d\tau\right)\,dt,
 \end{split}& j>k
 \end{cases}
 \end{equation}
where the prime over a sum means that the summation is taken over
$l\ne k$.

We put
\begin{equation}  \label{b9}
u_{jk}(x;\lambda)=
\begin{cases}
 \begin{split}
 &\int_0^x e^{i(b_j - b_k)\lambda(x-t)}\widetilde{q}_{jk}(t)\,dt,& j<k, \\
 -&\int_x^1 e^{i(b_j- b_k)\lambda(x-t)}\widetilde{q}_{jk}(t)\,dt,& j>k.
 \end{split}
\end{cases}
\end{equation}

Further, let
\begin{multline}  \label{b10}
A_{jkl}(\lambda)f(x):= \\
 =\begin{cases}
  \begin{split}
 &\int_0^x e^{i(b_j - b_k)\lambda(x-t)}\left(\widetilde{q}_{jl}(t)f(t)+
 \widetilde{q}_{jk}(t)\int_0^t \widetilde{q}_{kl}(\tau)f(\tau)\,d\tau\right)\,dt,& j<k
 \\
 -&\int_x^1e^{i(b_j - b_k)\lambda(x-t)}\left(\widetilde{q}_{jl}(t)f(t)+
 \widetilde{q}_{jk}(t)\int_0^t \widetilde{q}_{kl}(\tau)f(\tau)\,d\tau\right)\,dt,& j>k.
  \end{split}
\end{cases}
       \end{multline}
Clearly,  $A_{jkl}(\cdot): C[0,1]\to C[0,1]$ forms  the family of
continuous operators depending on $\lambda$ analytically.
Moreover, due to inequalities  \eqref{b2sect3},
$\|A_{jkl}(\lambda)\|=o(1)$ for $\lambda \in S,\ \lambda \to
\infty$.

The system~\eqref{b8} can be rewritten in the form
 \begin{equation}  \label{b11}
z_{jk}(x;\lambda)=u_{jk}(x;\lambda) + {\sum}'_{1\le l\le n}
A_{jkl}(\lambda) z_{lk}(t;\lambda),\quad j\ne k.
\end{equation}
Applying the method of successive approximations in the space
$C[0,1]\otimes\mathbb{C}^r$ to system~\eqref{b11} and using the
relation $\|A_{jkl}(\lambda)\|=o(1)$ we conclude that, for
sufficiently large $|\lambda|$,  $\lambda \in S$, the
system~\eqref{b11} has unique solution. Furthermore, the functions
$z_{jk}(x;\lambda)$ are analytic with respect to $\lambda\in S$,
and the following relations hold uniformly in $x\in [0,1]$
    \begin{equation}  \label{b12}
z_{jk}(x;\lambda)=u_{jk}(x;\lambda)(1+o(1)), \quad \lambda \in S,\
\lambda \to \infty, \quad j\in \{1,2,\dots,n\},\ \ j\ne k
    \end{equation}
The proof of this fact  is similar to that of ~\cite[Lemma
4.4.1]{Nai69}. Taking account of the relations
$u_{jk}(x;\lambda)=o(1)$ as  $\lambda\to\infty$, ~\eqref{b12} can
be rewritten as
\begin{equation}  \label{b13}
z_{jk}(x;\lambda)=o(1), \qquad \text {for} \quad \lambda \in
S,\quad  \lambda \to \infty, \qquad j\ne k.
\end{equation}
By substituting~\eqref{b13} into~\eqref{b7} we obtain
\begin{equation}  \label{b14}
z_{kk}(x;\lambda)= 1 + o(1), \qquad \lambda \in S,\quad  \lambda
\to \infty.
   \end{equation}
Next by substituting both~\eqref{b13} and~\eqref{b14}
in~\eqref{b6} we arrive at~\eqref{b4}.

It remains to note that, due to~\eqref{b4} for $x=0$, we have
\begin{equation}  \label{b15}
Y(x;\lambda)=(y_{jk}(x;\lambda))_{j,k=1}^n = I_n + o_n(1).
\end{equation}
Hence the system of solutions $Y_k(x;\lambda)$ is linearly
independent for  $\lambda \in S_R$ with sufficiently large $R$.
\end{proof}
      \begin{remark} \label{rem1}
Replacing  the condition $q_{jl}\in L^1(0,1)$ by the stronger
condition $q_{jl}\in L^\infty(0,1)$, we arrive at the  stronger
estimate
\begin{equation} \label{b16}
y_{jk}(x;\lambda)=\left(\delta_j^k+O\left(\tfrac1{|\lambda|}\right)\right)
e^{ib_k\lambda x}, \qquad \lambda \in S,\ \lambda \to \infty.
\end{equation}
However, the estimate~\eqref{b16} is false in general if only
$Q\in L^1(0,1)\otimes \Bbb C^{n\times n}.$ For instance, consider
the $2\times 2$ system \eqref{1.1} with $B= \diag(i,\ -i)$
  \begin{equation}
\begin{cases}
y_1'(x;\lambda)=\lambda y_1(x;\lambda) \\
y_2'(x;\lambda)=-\lambda
y_2(x;\lambda)+\frac{y_1(x;\lambda)}{\sqrt{1-x}}.
\end{cases}
  \end{equation}
Here  $q_{12}\equiv 0,$ $q_{21}=\tfrac 1{\sqrt{1-x}}\in L^1(0,1)$.
Using the estimate $\int_0^1e^{-2\lambda\tau^2}\, d\tau \sim
\frac1{\sqrt{\lambda}}$ it is not difficult to show that that the
estimate~\eqref{b16} is false.
          \end{remark}
%

\subsection{The minimality property }
Apparently the following statement is well known for experts. We
present it with the proof for completeness.
   \begin{lemma}\label{Lem_minimality}
let $T\in\frak S_{\infty}(\gH)$ and $\ker T=\{0\}$. Then the
system of EAF of the operator $T$ is  minimal.
   \end{lemma}
   \begin{proof}[Proof]
Let $\{\lambda_j\}^{\infty}_1$ be a system of eigenvalues of $T$
arranged in descending order of their modulus:
    \begin{equation}
|\lambda_1|\ge|\lambda_2|\ge \ldots \ge |\lambda_k|\ge|\lambda_{k
+ 1}|\ge\ldots >0.
    \end{equation}
Denote by $\frak N_j(T):=\frak N_{\lambda_j}(T)$ the corresponding
root subspaces of $T$. It is easily  seen that
      \begin{equation}\label{2.25A}
\frak N_j(T)\perp\frak N_k(T^*)\quad\text{for}\quad  j\not = k.
      \end{equation}
Moreover, by Fredholm theorem,  $\dim\frak N_j(T)=\dim\frak
N_j(T^*),\  j\in{\Bbb N},$ since $\lambda_j\not = 0.$ Further, let
$\{e_{jp}\}^{n_j}_{p=1}$ and $\{f_{jk}\}^{n_j}_{k=1}$ be the
basses in $\frak N_j(T)$ and $\frak N_j(T^*)$, respectively. Then
the "Gram matrix"
    \begin{equation*}
G_j:=(\langle e_{jp},f_{jk}\rangle)^{n_j}_{p,k=1}
    \end{equation*}
is non-singular. Assuming the contrary we find a non-zero vector
$f=\sum^{n_j}_{k=1}a_kf_{jk}\in\frak N_j(T^*)$ which is orthogonal
to $\frak N_j(T)$. Thus, due to \eqref{2.25A}
$$
f\perp\frak H_1 := \Span\{\frak N_k(T):\ k\in\Bbb N\},
$$
i.e. $f\in\frak H_2 :=\frak H^{\perp}_1$. Let $P_2$ be an
orthogonal projection on the subspace $\frak H_2$. By \cite[Lemma
1.4.2]{GohKre65} the operator $T_2=P_2 T P_2$ is volterra
operator, hence so is the adjoint operator $T^*_2=P_2 T^* P_2$.
Since $f \in \frak N_j(T^*),$ we can find  $k< n_j$ such that
$u:=(T^*-\overline{\lambda}_j)^k f\not = 0$  and $T^*u =
\overline{\lambda}_ju$. Since $\frak H_2$ is an invariant subspace
for $T_2^*$,  $u \in \frak H_2$ and  $T^*_2u = T^*u =
\overline{\lambda}_j u$ where  ${\lambda}_j \ne 0$. This
contradiction shows that  the matrix $G_j$ is non-singular.

Thus,  the basis $\{f_{jk}\}^{n_j}_{k=1}$ in  $\frak N_j(T^*)$ can
be chosen to be  biorthogonal  to the basis
$\{e_{jp}\}^{n_j}_{p=1}$, i.e. to satisfy $\langle e_{jp},
f_{jk}\rangle=\delta_{pk}, \  p,k\in \{1, \ldots, {n_j}\}.$
Consider the union of both systems. Then using the latter
identities and \eqref{2.25A} we obtain two biorthogonal systems.
Thus, the system $\cup_{j=1}^{\infty}\{e_{jp}\}^{n_j}_{p=1}$  is
minimal.
      \end{proof}

\section{Completeness of the root functions of BVP for first order-systems}

\subsection{Proof of the main result}

Here we present the proof of Theorem 1.1. On the second step  we
use the idea of reduction of the proof  of completeness of  the
BVP~\eqref{1.1}, \eqref{1.3} to the investigation of that for
solutions to the (incomplete) Cauchy problem. The idea of such
reduction goes back to the paper by A.A.~Shkalikov~\cite{Shk76}
where it was applied to BVP for $n$th order differential
equations.

\begin{proof}[Proof of Theorem 1.1]
(i)  Suppose that $\Phi(x;\lambda)$ is a fundamental $n\times n$
matrix solution of equation~\eqref{1.1} corresponding to the
initial condition
   \begin{equation}\label{2.7}
\Phi(0;\lambda)=I_n.
    \end{equation}
Further, denote by $\Phi_j(x;\lambda)$ the $j$th vector column of
the matrix $\Phi(x;\lambda)$, i.e.,
    \begin{equation}\label{2.8}
\Phi(x;\lambda)=(\Phi_1,\dots,\Phi_n),\quad
\Phi_j(x;\lambda)=\col(\varphi_{1j},\dots,\varphi_{nj}).
    \end{equation}
It is clear that the general solution of equation~\eqref{1.1} is
of the form
    \begin{equation}\label{2.9}
U(x;\lambda)=\sum_{j=1}^n\alpha_j(\lambda)\Phi_j(x;\lambda),\qquad
\alpha_j(\lambda)\in\mathbb{C}.
    \end{equation}
By substituting~\eqref{2.9} into~\eqref{1.3} we derive to the
equation for eigenvalues and eigenfunctions of
problem~\eqref{1.1}, \eqref{1.3}:
\begin{multline}\label{2.10}
C\sum_{j=1}^n\alpha_j(\lambda)\Phi_j(0;\lambda)+
D\sum_{j=1}^n\alpha_j(\lambda)\Phi_j(1;\lambda)=\\=
(C\Phi(0;\lambda)+D\Phi(1;\lambda))
  \begin{pmatrix}
    \alpha_1 \\
    \ldots\\
    \alpha_n \\
  \end{pmatrix}
=(C+D\Phi(1;\lambda))
  \begin{pmatrix}
    \alpha_1 \\
    \ldots\\
    \alpha_n \\
  \end{pmatrix}=0.
\end{multline}
The equation~\eqref{2.10} has nontrivial solution if and only if
the matrix $A_\Phi(\lambda):=(C+D\Phi(1;\lambda))$ is singular,
i.e., if
    \begin{equation}\label{2.11}
\Delta_\Phi(\lambda):=\det A_\Phi(\lambda):=\det
(C+D\Phi(1;\lambda))=0.
    \end{equation}
It follows that the spectrum $\sigma(L_{C,D})$ of
problem~\eqref{1.1}, \eqref{1.3} coincides with the roots of the
characteristic determinant $\Delta_\Phi(\cdot)$. In what follows
we will show that the assumption (b) of the theorem yields the
nondegeneracy of the $\Delta_\Phi(\lambda)$, i.e., the relation
$\Delta_\Phi(\lambda)\not\equiv 0$. Therefore, the spectrum
$\sigma(L_{C,D})$ of problem~\eqref{1.1}, \eqref{1.3} is discrete,
i.e., $\sigma(L_{C,D})=:\{\lambda_k\}_1^\infty$.

Denote by
$\widetilde{A}_\Phi(\lambda)=(\Delta_{jk}(\lambda))_{j,k=1}^n$ the
matrix associated to $A_\Phi(\lambda)$, and introduce the vector
functions
\begin{equation}\label{2.13}
U_j(x;\lambda):=\sum_{k = 1}^n
\Delta_{jk}(\lambda)\Phi_k(x;\lambda),\qquad j\in\{1,2,\dots,n\}.
\end{equation}
Here two cases are possible: $U_j(x;\lambda_k)\ne 0$ and
$U_j(x;\lambda_k)=0$.  If $U_j(x;\lambda_k)\ne 0$ then
relations~\eqref{2.10}, \eqref{2.11} and~\eqref{2.13} together
imply that $U_j(x;\lambda_k)$ is an eigenfunction of
problem~\eqref{1.1}, \eqref{1.3} corresponding to the eigenvalue
$\lambda_k$.

Moreover, if $\lambda_k$ is an $m_k$-multiple ($m_k>1$) zero of
the function $\Delta(\lambda):=\Delta_\Phi(\lambda)$, then the
vector functions
\begin{equation}\label{2.14}
\frac 1{p!} \left. D_\lambda^p U_j(x;\lambda)
\right|_{\lambda=\lambda_k}:= \frac 1{p!} \left.
\frac{\partial^p}{\partial\lambda^p} U_j(x;\lambda)
\right|_{\lambda=\lambda_k}, \qquad p\in\{0,1,\dots,m_k-1\},
\end{equation}
form a chain of an eigenfunction and associated functions of
problem~\eqref{1.1}, \eqref{1.3} corresponding to the eigenvalue
$\lambda_k$. Indeed, we have
\begin{eqnarray}\label{2.15}
\frac 1{p!} \left. L D_\lambda^p U_j(x;\lambda)
\right|_{\lambda=\lambda_k}= \frac 1{p!} \left. D_\lambda^p
LU_j(x;\lambda) \right|_{\lambda=\lambda_k}=
  \frac 1{p!} \left. D_\lambda^p (\lambda U_j(x;\lambda)) \right|_{\lambda=\lambda_k} \nonumber  \\
 = \frac 1{p!} \left. \lambda_k D_\lambda^p U_j(x;\lambda) \right|_{\lambda=\lambda_k}+
  \frac 1{(p-1)!} \left. D_\lambda^{p-1} U_j(x;\lambda) \right|_{\lambda=\lambda_k}.
\end{eqnarray}
Besides, both~\eqref{2.10} and~\eqref{2.13} yield that $\left.
D_\lambda^p U_j(x;\lambda) \right|_{\lambda=\lambda_k}$ satisfies
the boundary condition~\eqref{1.3}.  For instance, in the case
$p=1$, this is implied by the relation
\begin{equation}\label{2.16}
(C+D\Phi(1;\lambda_k))
  \begin{pmatrix}
    \Delta_{11}'(\lambda_k) \\
    \ldots\\
    \Delta_{1n}'(\lambda_k)
  \end{pmatrix}
  + (C+D\Phi(1;\lambda_k))'
  \begin{pmatrix}
    \Delta_{11}(\lambda_k) \\
    \ldots\\
    \Delta_{1n}(\lambda_k)
  \end{pmatrix}=0,
\end{equation}
which holds for $\lambda=\lambda_k$ if
$\Delta(\lambda_k)=\Delta'(\lambda_k)=0$.

Now let $U_j(x;\lambda_k)=0$. As above, we consider the sequence
of the vector functions $\displaystyle \left. D_\lambda^p
U_j(x;\lambda) \right|_{\lambda=\lambda_k},$\ \
$p\in\{0,1,\dots,m_k-1\}$. Let $s$ stand for the minimal number
$p$ such that $\left. L D_\lambda^p U_j(x;\lambda)
\right|_{\lambda=\lambda_k}\ne 0$, i.e.,
\begin{equation}
\begin{cases}
\left. D_\lambda^p U_j(x;\lambda) \right|_{\lambda=\lambda_k}=0
\quad
\text{ for } p\in\{0,1,\dots,s-1\}; \\
 \left. D_\lambda^s U_j(x;\lambda) \right|_{\lambda=\lambda_k}\ne 0.
\end{cases}
\end{equation}

In this case, we obtain:
   \begin{multline}
 \frac 1{s!} \left. L D_\lambda^s U_j(x;\lambda)
\right|_{\lambda=\lambda_k}=
 \frac 1{s!} \left. D_\lambda^s L U_j(x;\lambda) \right|_{\lambda=\lambda_k}=
  \frac 1{s!} \left. D_\lambda^s (\lambda_k U_j(x;\lambda)) \right|_{\lambda=\lambda_k}   \\
 \!=\!\frac 1{s!} \left. \lambda_k D_\lambda^s U_j(x;\lambda) \right|_{\lambda=\lambda_k}+
 \frac 1{(s-1)!} \left. D_\lambda^{s-1} U_j(x;\lambda)
 \right|_{\lambda=\lambda_k}\!=\!
 \frac 1{s!} \left. \lambda_k D_\lambda^s U_j(x;\lambda) \right|_{\lambda=\lambda_k},
\end{multline}
since $\left. D_\lambda^{s-1} U_j(x;\lambda)
\right|_{\lambda=\lambda_k}=0$. Hence for $s<m_k$ the sequence of
the vector functions $\left. D_\lambda^s U_j(x;\lambda)
\right|_{\lambda=\lambda_k},\dots, \left. D_\lambda^{m_k}
U_j(x;\lambda) \right|_{\lambda=\lambda_k}$ forms a chain of an
eigenfunction and associated functions of problem~\eqref{1.1},
\eqref{1.3} corresponding to the eigenvalue $\lambda_k$. In this
case, the fulfilment of the boundary conditions is verified as
above.

Thus,  the system of functions $\{\left. D_\lambda^p
U_j(x;\lambda) \right|_{\lambda=\lambda_k}\}_{p=0}^{m_k-1}$ is
either zero, or it span the root subspace of the operator
$L_{C,D}$  corresponding to  $\lambda_k.$

(ii) In this step we reduce the problem \eqref{1.1}--\eqref{1.3}
to similar problem with a potential matrix
$Q(\cdot)=\bigl(q_{jk}(\cdot)\bigr)^r_{j,k=1}$ having  zero
diagonal, i.e. $q_{jj}(\cdot)=0,\  j\in\{1,\ldots, r\}$. It will
allow us to apply Proposition \ref{th2.1}.

To this end we denote by $W(\cdot)$ the fundamental $n\times n$
matrix solution of the Cauchy problem
    \begin{equation}\label{3.1B}
iB W'(x) = Q_1(x)W(x), \qquad  W(0)=I_n.
    \end{equation}
where the $n\times n$ matrix function $Q_1(\cdot)$ is quasidiagonal with blocks
$q_{jj}(\cdot)$,
    \begin{equation}\label{3.2B}
Q_1(x) = \diag\bigl(q_{11}(x),\ldots, q_{rr}(x)\bigr).
    \end{equation}
Since $BQ_1(x)=Q_1(x)B$ for any $x\in[0,1]$, the matrix functions
$W_1(\cdot)= B W(\cdot)$ and $W_2(\cdot) = W(\cdot)B$ satisfy
equation \eqref{3.1B} and common initial conditions
   \begin{equation}
iB W'_j(x)= Q_1(x)W_j(x), \qquad   W_j(0)=B, \quad j\in \{1,2\}.
  \end{equation}
According to the Cauchy uniqueness theorem $W_1(x)=W_2(x)$ for
$x\in[0,1]$, i.e.
   \begin{equation}\label{3.3B}
W(x)B - B W(x) = 0, \qquad  x\in[0,1].
   \end{equation}

Letting $\widetilde L=(I\otimes W)^{-1}L(I\otimes W)$ we deduce
from \eqref{1.1}, \eqref{3.1B} and \eqref{3.3B} that for any $f\in
C^1[0,1]\otimes{\Bbb C}^n$
    \begin{eqnarray}\label{3.16B}
\widetilde L f -\lambda f= W^{-1}(x)(-iB)W(x) f' + W^{-1}(x)(-iB)W'(x) f  \nonumber \\
+ W^{-1}(x)Q(x)W(x)f -\lambda f = -i B\frac{d}{dx}f + {\widetilde
Q}(x)f -\lambda f,
    \end{eqnarray}
where
   \begin{equation}\label{3.4B}
\widetilde Q(x) := W^{-1}(x)\bigl(Q(x)-Q_1(x)W(x)\bigr).
   \end{equation}
It follows from \eqref{3.3B} that the matrix function $W(\cdot)$
is quasidiagonal,
    \begin{equation}\label{3.5B}
W(x) = \diag\bigl(W_{11}(x),\ldots, W_{rr}(x)\bigr),
  \end{equation}
with $n_j\times n_j$ nonsingular matrix blocks $W_{jj}(\cdot),\
j\in\{1,\ldots,r\}$. It follows from \eqref{3.4B} and \eqref{3.5B}
that $\wt Q(\cdot)$ is of the form
     \begin{equation}
\wt Q(x) = \bigl(\wt Q_{jk}(x)\bigr)^r_{j,k=1}, \quad \wt
Q_{jj}(x) = 0,\qquad  x\in[0,1],\quad  j\in\{1,\ldots, r\}.
  \end{equation}
Thus, the problem \eqref{1.1}, \eqref{1.3} transforms into similar
problem for equation \eqref{3.16B} with $\wt Q(\cdot)$ instead of
$Q(\cdot)$ and the boundary conditions
    \begin{equation}
C_1 y(0)+D_1 y(1)=0
  \end{equation}
in place of \eqref{1.3}. Here $C_1 := C W(0) = C$ and $D_1 :=
DW(1)$. Due to the block structure \eqref{3.5B} of $W(\cdot)$ and
conditions $\det W_{jj}(\cdot)\not =0$ the pairs $\{C,D\}$ and
$\{C,D W(1)\}$ satisfy the conditions of Theorem \ref{th2.1} only
simultaneously.

Thus, in what follows without loss of generality we may assume
that the matrix function
$Q(\cdot)=\bigl(q_{jk}(\cdot)\bigr)^r_{j,k=1}$ has zero diagonal,
i.e. $q_{jj}(\cdot)=0, j\in\{1,\ldots, r\}$.

(iii) We prove the completeness of system~\eqref{2.14} by
contradiction. To this end, we assume that there exists a vector
function $f=\col(f_1,\dots, f_n) \in L^2[0,1] \otimes
\mathbb{C}^n$ orthogonal to this system. Consider  the entire
function
\begin{equation}\label{2.17}
F_1(\lambda):=(U_1(x;\lambda),f(x))_{L^2[0,1]\otimes\mathbb{C}^n}=
\sum_{j=1}^n\Delta_{1j}(\lambda)
\int_0^1\bigl\langle\Phi_j(x;\lambda),f(x)\bigr\rangle\,dx.
\end{equation}
Clearly, any  $\lambda_k(\in\sigma(L_{C,D}))$ is the zero of
$F_1(\cdot)$ of multiplicity at least $m_k$, i.e.,
\begin{equation}\label{2.18}
\left. F_1^{(p)}(\lambda) \right|_{\lambda=\lambda_k}=0, \qquad
p\in\{0,1,\dots,m_k-1\},\quad \lambda_k\in\sigma(L_{C,D}).
\end{equation}
Thus, the ratio
\begin{equation}\label{fun_G}
 G_1(\lambda):=\frac{F_1(\lambda)}{\Delta(\lambda)}
\end{equation}
is  an entire function.  Let us prove that $G_1(\lambda)\equiv 0$
by estimating its growth.

To this end  we obtain another representation of $G_1(\cdot)$
which is more convenient  for the estimation. Moreover, to
simplify the notions, we restrict ourselves to the case $r=n$,
i.e., assume  that the spectrum of the matrix $B$ is simple.

As in Proposition~\ref{BirkSys}, the complex plane can be divided
into the sectors $S_p = \{z\in\mathbb{C}\,:\ \phi_p < \arg
z<\phi_{p + 1}\}$ such that, for all $\lambda$ inside a certain
sector, the numbers $b_j$ can be ordered as
   \begin{equation}\label{2.19}
\Re(ib_1\lambda)<\dots<\Re(ib_\varkappa\lambda)<0<
\Re(ib_{\varkappa+1}\lambda)<\dots<\Re(ib_n\lambda).
\end{equation}
Moreover, for a sufficiently large $R>0$, in the domain
\begin{equation}\label{2.20}
S_{p,\varepsilon,R} := \{\lambda\in\mathbb{C}_+:\ \phi_p +
\varepsilon<\arg \lambda<\phi_{p + 1}-\varepsilon,\ \
|\lambda|>R\},
   \end{equation}
there exist $n$ linearly independent solutions $Y_j(x;\lambda) =
\col(y_{1j},\dots,y_{nj})$ analytic with respect to $\lambda$ and
having the following asymptotic behavior
\begin{equation}\label{2.21}
y_{jk}(x;\lambda)=\left(\delta_k^j+o\left(1\right)\right)e^{ib_j\lambda x},
\qquad \lambda\in S_{p,\varepsilon,R},
\end{equation}
uniform with respect to $x\in[0,1]$.

Since the solutions $Y_j(\cdot;\lambda)\quad(1\le j\le n)$ are
linearly independent for any $\lambda\in S_{p,\varepsilon,R}$,
then the fundamental $n\times n$ matrices  $\Phi(x;\lambda)$ and
$Y(x;\lambda):=(Y_1,\dots,Y_n)$ of the system~\eqref{1.1} are
related by
   \begin{equation}\label{2.22}
\Phi(x;\lambda)=Y(x;\lambda)P(\lambda),\qquad x\in [0,1], \quad
\lambda\in S_{p,\varepsilon,R},
   \end{equation}
where  $P(\lambda)=:(p_{kj}(\lambda))_{k,j=1}^n$ is an analytical
invertible matrix function in $S_{p,\varepsilon,R}$.

Further, apart from $A_\Phi(\lambda)$, we introduce the matrix
function
\begin{equation}\label{2.23}
A_Y(\lambda) = CY(0;\lambda) + DY(1;\lambda),
\end{equation}
and denote its determinant  by $\Delta_Y(\lambda):=\det
A_Y(\lambda)$. Besides this, alongside with $U_j(x;\lambda)$ of
the form~\eqref{2.13}, we consider the vector functions
\begin{equation}\label{2.24}
V_j(x;\lambda):=\sum_{k=1}^n\Delta_Y^{jk}(\lambda)Y_k(x;\lambda),
\qquad j\in\{1,2,\dots,n\},
\end{equation}
where $\Delta_Y^{jk}(\lambda)$ is the  cofactor  of the $jk$th
entry of the matrix $A_Y(\lambda)$. Clearly, $V_j(x;\lambda)$ are
holomorphic in $S_{p,\varepsilon,R}$.

Both~\eqref{2.22}, \eqref{2.23} and the definition of
$A_\Phi(\lambda)$ (see~\eqref{2.11}) yield the relations
\begin{equation}\label{2.25}
A_\Phi(\lambda)=A_Y(\lambda)P(\lambda),
\qquad\Delta_\Phi(\lambda)=\Delta_Y(\lambda)\,\det P(\lambda).
\end{equation}

Let $A_\Phi(\lambda)=:(a_{jk}(\lambda))_{j,k=1}^n,\quad
A_Y(\lambda)=:(\widetilde a_{jk}(\lambda))_{j,k=1}^n$. Taking
account of these notation, we derive from~\eqref{2.22}
and~\eqref{2.25} the relations
   \begin{equation}\label{2.26}
   \begin{pmatrix}
    \varphi_{j1} & \varphi_{j2} & \dots & \varphi_{jn} \\
     a_{21} &  a_{22} & \dots &  a_{2n} \\
    \ldots\\
     a_{n1} &  a_{n2} & \dots &  a_{nn} \\
  \end{pmatrix}=
  \begin{pmatrix}
     y_{j1}       &    y_{j2}     & \dots &  y_{jn}       \\
    \widetilde a_{21} & \widetilde a_{22} & \dots & \widetilde a_{2n} \\
    \ldots\\
    \widetilde a_{n1} & \widetilde a_{n2} & \dots & \widetilde a_{nn} \\
  \end{pmatrix}P(\lambda),\quad  j\in\{1,\dots,n\}.
\end{equation}
Note that the system~\eqref{2.26} is equivalent to the formal
equality that can be obtained from the first equation
in~\eqref{2.25} if one replaces the first lines in the matrices
$A_\Phi(\lambda)$ and $A_Y(\lambda)$ by the "lines"\
$(\Phi_1,\dots,\Phi_n)$ and $(Y_1,\dots,Y_n)$, respectively. The
desirable connection between the vector functions $U_1(x;\lambda)$
and $V_1(x;\lambda)$ is implied now by~\eqref{2.13}, \eqref{2.24}
and~\eqref{2.26}:
\begin{equation}\label{2.27}
U_1(x;\lambda)=V_1(x;\lambda)\det P(\lambda),\qquad\lambda\in
S_{p, \varepsilon,R}.
     \end{equation}
By setting
     \begin{multline}\label{2.28}
\widetilde F_1(\lambda):=(V_1(x;\lambda),
f(x))_{L^2\otimes\mathbb{C}^n}
 =\sum_{j=1}^n \Delta^{1j}_Y(\lambda)\int_0^1\left\langle
Y_j(x;\lambda),f(x)\right\rangle\,dx \\
= \sum_{j=1}^n
\Delta^{1j}_Y(\lambda)\sum_{k=1}^n\int_0^1Y_{kj}(x;\lambda)\overline{f_k(x)}\,dx
      \end{multline}
and by taking into account~\eqref{2.17}, \eqref{2.27}
and~\eqref{2.28}, we arrive at the relation
      \begin{equation}\label{2.29}
F_1(\lambda)=\widetilde F_1(\lambda)\det P(\lambda).
   \end{equation}
Finally, combining the second equality in~\eqref{2.25}
with~\eqref{2.29}, we arrive at the second representation of the
entire function $G_1(\cdot)$:
\begin{equation}\label{2.30}
G_1(\lambda) = \widetilde
F_1(\lambda)/\Delta_Y(\lambda),\qquad\lambda\in
S_{p,\varepsilon,R}.
\end{equation}

(iv) In  this step we estimate $G_1(\cdot)$ on the rays $l_m
=\{\zeta_m t:\ t\in \mathbb{R}_+\}$,  \ $m\in \{1,2,3\},$ using
the representation~\eqref{2.30}. Here $\zeta_m=i z_m$ where $z_m$
are taken from the condition $(b)$ of the theorem.

Since $C= (c_{kj})_{k,j=1}^n,\  D = (d_{kj})_{k,j=1}^n$, it
follows from \eqref{2.23} and  \eqref{2.21}  that the matrix
$A_Y(\lambda)$ admits the following representation
\begin{equation}\label{3.31}
A_Y(\zeta_mt)\!=\!
\begin{pmatrix}
[c_{11}]\!+\![d_{11}]e^{ib_1\zeta_mt} & [c_{12}]\!+\![d_{12}]e^{ib_2\zeta_mt} & \!\dots\! & [c_{1n}]\!+\![d_{1n}]e^{ib_n\zeta_mt} \\
[c_{21}]\!+\![d_{21}]e^{ib_1\zeta_mt} & [c_{22}]\!+\![d_{22}]e^{ib_2\zeta_mt} & \!\dots\! & [c_{2n}]\!+\![d_{2n}]e^{ib_n\zeta_mt} \\
\hdotsfor{4} \\
[c_{n1}]\!+\![d_{n1}]e^{ib_1\zeta_mt} & [c_{n2}]\!+\![d_{n2}]e^{ib_2\zeta_mt} &
\!\dots\! & [c_{nn}]\!+\![d_{nn}]e^{ib_n\zeta_mt}
\end{pmatrix}\!.
\end{equation}
Noting  that
$$
[c_{kj}]\!+\![d_{kj}]e^{ib_j\zeta_mt} \sim c_{kj} \qquad \text{for}\quad \Re (
b_j z_m)>0, \quad k\in \{1,\dots,n\},
$$
and
$$
[c_{kj}]\!+\![d_{kj}]e^{ib_j\zeta_mt}  \sim d_{kj}e^{ib_j\zeta_mt} \qquad
\text{for}\quad  \Re (b_j z_m) < 0, \quad k\in \{1,\dots,n\},
$$
we arrive at the  asymptotic estimate for the characteristic determinant
\begin{equation}\label{3.32}
\Delta_Y(\zeta_mt) =  \det A_Y(\zeta_mt) = e^{\beta_m t}(\det
T_{z_m B}(C,D) + o(1))\quad  \text{as} \quad t \to \infty,
\end{equation}
along the ray $l_m$. Here  $\beta_m := \sum_{\Re (i b_j
\zeta_m)>0}i b_j \zeta_m$ and $T_{z_m B}(C,D)$ is the matrix from
the assumption (b) of the theorem.

Next we estimate $\widetilde F_1(\cdot).$ Since
$\Delta_Y^{1j}(\zeta_mt) = O(e^{\beta_m t})$ for $\Re (i b_j
\zeta_m)<0$, estimates ~\eqref{2.21} yield
\begin{equation}\label{3.33}
\Delta_Y^{1j}(\zeta_mt)Y_j(x;\zeta_mt)=e^{\beta_m t}O(e^{i b_j
\zeta_m tx}).
\end{equation}
If $\Re (i b_j \zeta_m)>0$ then
$\Delta_Y^{1j}(\zeta_mt)=O(e^{(\beta_m-i b_j \zeta_m)t})$, and in
this case we obtain:
\begin{equation}\label{3.34}
\Delta_Y^{1j}(\zeta_mt)Y_j(x;\zeta_mt)=e^{(\beta_m-i b_j
\zeta_m)t}O(e^{i b_j \zeta_m tx})=e^{\beta_m t}O(e^{i b_j \zeta_m
t(x-1)}).
\end{equation}

Denote by $s_-$ the maximal negative number from $\Re (i b_j \zeta_m)$, and by
$s_+$ the minimal positive number from the same set. Then we have
\begin{equation}\label{3.35}
\Delta_Y^{1j}(\zeta_mt)Y_j(x;\zeta_mt)=e^{\beta_m t}O(\max (e^{s_-
tx},e^{s_+ t(x-1)})), \qquad j\in\{1,2,\dots,n\}.
\end{equation}

Hence the function $V_1$ of the form~\eqref{2.24} are estimated
along the rays $l_m = \{\lambda: \lambda=\zeta_mt\}$,  as above,
i.e.,
\begin{equation}\label{3.36}
V_1(x;\zeta_mt)=e^{\beta_m t}O(\max (e^{s_- tx},e^{s_+
t(x-1)}))=e^{\beta_m t}O(e^{s_- tx}+e^{s_+ t(x-1)}).
\end{equation}
It follows that
\begin{multline}\label{3.37}
 \widetilde{F_1}(\zeta_mt)=\int_0^1\left\langle V_1(x;\zeta_mt),f(x)\right\rangle \,dx =
 e^{\beta_m t} O \left(\int_0^1|f(x)|(e^{s_- tx}+e^{s_+ t(x-1)})\,dx \right) \\
 \le  Ce^{\beta_m t}\sqrt{\int_0^1|f(x)|^2\,dx}\sqrt{\int_0^1(e^{s_- tx}+e^{s_+ t(x-1)})^2\,dx}=o(e^{\beta_m t}),
\end{multline}
since $\int_0^1(e^{s_- tx}+e^{s_+ t(x-1)})^2\,dx\to 0$ as
$t\to\infty$.

Combining estimates \eqref{3.32} and \eqref{3.37} we get
\begin{equation*}
G_1(\zeta_mt) = \frac{\widetilde{F_1}(\zeta_mt)}{\Delta_Y(\zeta_mt)} =
\frac{o(e^{\beta_m t})}{(\det T_{z_m B}(C,D) + o(1))e^{\beta_m t}}\to 0 \quad
\text{as}\quad t\to\infty.
\end{equation*}
It follows from \eqref{fun_G}, \eqref{2.17}, that  $G_1(\cdot)$ is
the entire function of type not greater than exponential, hence it
is bounded in each of the (convex) angles formed by pairs of the
rays $l_k$. Since the origin is the interior point of the triangle
$\triangle_{\zeta_1 \zeta_2 \zeta_3}$, we obtain that these angles
cover the whole complex plain. Thus, $G_1(\cdot)$ is bounded in
$\mathbb{C}$ and tends to zero along each of the rays $l_k$, Hence
$G_1(\lambda)\equiv 0$, by the Liouville theorem.

As in~\eqref{fun_G}, we introduce the functions
\begin{equation}
G_j(\lambda):= {F_j(\lambda)}{\Delta(\lambda)}^{-1},\qquad
j\in\{2,3,\dots,n\},
\end{equation}
and show that $G_j(\lambda)\equiv 0$ for $j\in\{2,3,\dots,n\}$.

(v) Note that, for $\lambda\notin\sigma(L_{C,D})$, the functions
$U_j(x;\lambda)$ form the fundamental systems of solutions of the
system~\eqref{1.1}. Since $f(x)$ is orthogonal to all the
$U_j(x;\lambda)$, $\ j\in\{1,2,\dots,n\}$, we conclude that it is
orthogonal to all solutions of the system~\eqref{1.1} whenever
$\lambda\notin\sigma(L_{C,D})$. Therefore,
\begin{equation}\label{3.38}
\int_0^1\bigl\langle\Phi_j(x;\lambda),f(x)\bigr\rangle\,dx=0,\qquad
\lambda\notin\sigma(L_{C,D}),\quad j\in\{1,2,\dots,n\}.
\end{equation}
But, due to the continuity of the integral~\eqref{3.38} with
respect to $\lambda$ and the discreteness of the set
$\sigma(L_{C,D})$, the following relations hold:
\begin{equation}\label{2.42}
\int_0^1\bigl\langle\Phi_j(x;\lambda),f(x)\bigr\rangle\,dx \equiv
0,\qquad \lambda\in\mathbb{C},\quad j\in\{1,2,\dots,n\}.
\end{equation}

(vi) At this step, we  show that the vector function $f$
satisfying relations~\eqref{2.42} is the zero function. To this
end, consider the resolvent $R_L(\lambda )$ of the operator $L$ of
the form~\eqref{1.1} subject to the initial conditions
   \begin{equation} \label{2.43}
Y(0)=\col\bigl(y_1(0),\dots,y_n(0)\bigr)=0.
    \end{equation}
As above, let $\Phi (x;\lambda)$ stand for the fundamental matrix
solution of the equation~\eqref{1.1} satisfying the condition
\eqref{2.7}. It can easily be seen that the Green matrix of the
Cauchy problem~\eqref{1.1}, \eqref{2.43} is
  \begin{equation} \label{2.44}
G(x,t;\lambda )=
\begin{cases}
\Phi (x;\lambda )\Phi ^{-1}(t;\lambda )(-iB)^{-1},& t\le x \\
0 , &         t>x
\end{cases},
   \end{equation}
and is an entire function with respect to $\lambda \in \mathbb{C }$. Hence
$R_L(\lambda )$ is a Volterra operator: $(R_L(\lambda)\varphi)(x) = \int
^x_0G(x,t;\lambda )\varphi(t)dt,\quad \varphi \in L^2[0,1]$.

Alongside with the $\Phi (x;\lambda)$, consider the matrix
function
\begin{equation*}
 Y(x;\lambda ):=\bigl(Y_1(x;\lambda),...,Y_n(x;\lambda )\bigr)
\end{equation*}
consisting from the solutions $Y_j(x;\lambda
)=\col(y_{1j},...,y_{nj})$ satisfying the asymptotic relations
 ~\eqref{2.21}. Clearly, $Y(x;\lambda )$ is the fundamental matrix
of~\eqref{1.1} for $\lambda \in S_{\pm}:=\pm S_{p,\varepsilon ,R}$. By
\eqref{2.22}
$ Y(x;\lambda )=\Phi (x;\lambda )P^{-1}(\lambda ), \ \lambda \in
S_{\pm},$ where $P^{-1}(\lambda )\in \mathbb{C}^{n\times n}$ for
$\lambda \in S_{\pm}$. Therefore,
   \begin{equation*}
Y(x;\lambda )Y^{-1}(t;\lambda )(-iB)^{-1}= \Phi (x;\lambda )\Phi
^{-1}(t;\lambda )(-iB)^{-1}, \quad \lambda \in S_{\pm},
   \end{equation*}
and the Green matrix $G(x,t;\lambda )$ is the analytic continuation of the
matrix function $Y(x;\lambda )Y^{-1}(t;\lambda )(-iB)^{-1}$.
In particular, for $\lambda \in S_{\pm}$ the operator
$R^*_L(\lambda):=\bigl(R_L(\lambda)\bigr)^*$ admits the
representation
   \begin{equation} \label{2.45}
(R^*_L(\lambda)\varphi)(x)=(iB^{-1})^* Y^{-1}(x,\lambda)^* \int ^1_x Y^* (t,
\lambda)\varphi(t)dt, \quad \lambda \in S_{\pm}.
   \end{equation}
Further, since $f$ satisfies conditions~\eqref{2.42}, we have
    \begin{equation} \label{2.46}
 \int ^1_0Y^*(t,\lambda)f(t)dt=0, \qquad \lambda \in S_{\pm}.
    \end{equation}
From~\eqref{2.21} follows that $Y(x;\lambda )$ admits the representation
   \begin{equation} \label{2.47}
Y(x;\lambda )=\mathcal{I}_n(x;\lambda ) e(x;\lambda ), \qquad \lambda\in
S_{\pm},
    \end{equation}
in which $\mathcal{I}_n(x;\lambda )=I_n+o_n(1)$ and
   \begin{equation} \label{2.48}
e(x;\lambda ):=\diag(e^{ib _1\lambda x},...,e^{ib _n\lambda x}).
    \end{equation}
By  multiplying~\eqref{2.46} from the left by the matrix
\begin{equation*} \widetilde e(x;\lambda
):=\diag(e^{i\,\overline{b_1 \lambda} x},...,e^{i\,\overline{b_n \lambda}
x})=e^{-1}(x;\lambda)^*
    \end{equation*}
and by taking into account~\eqref{2.47} and~\eqref{2.48}, we
arrive at the relation
   \begin{multline} \label{2.49}
\Theta (x;\lambda ):=\int^1_x\widetilde e(x-t;\lambda )
 \mathcal{I}_n^*(t;\lambda) f(t)dt=\\
=-\int^x_0\widetilde e(x-t;\lambda ) \mathcal{I}_n^*(t;\lambda) f(t)dt, \qquad
\lambda\in S_{\pm}.
   \end{multline}
By setting
  \begin{equation} \label{2.50}
g(t;\lambda )=\col\bigl(g_1(t;\lambda ),...,g_n(t;\lambda )\bigr):=
\mathcal{I}_n^*(t;\lambda) f(t),\quad \lambda \in S_{\pm},
  \end{equation}
we rewrite the matrix equality~\eqref{2.49} as a system of $n$
scalar equalities:
  \begin{multline} \label{2.51}
\int ^x_0e^{i\,\overline{b_j \lambda} (x-t)}g_j(t;\lambda )dt=
-\int ^1_xe^{i\,\overline{b _j \lambda} (x-t)}g_j(t;\lambda)dt, \\
\qquad \lambda \in S_{\pm}, \quad j\in\{1,2,\dots,n\}.
   \end{multline}

Since $\Re(i \overline{b_j \lambda})=-\Re(i b_j \lambda)$ then~\eqref{2.19}
implies that the functions $e^{i\,\overline{b_j \lambda} x}$, $\left(x\in
[0,1]\right)$ are bounded in the sector $S_-$ for $j\in \{1,...,\kappa\}$ and
in the sector $S_+$ for $j\in \{\kappa+1,...,n\}$. Due to~\eqref{2.50} the
functions $g_j(\cdot;\lambda )$ have uniformly bounded norms in $L^2[0,1]$ for
$\lambda \in S_{\pm}$. Now we conclude from~\eqref{2.51} that
\begin{equation}\label{Theta(x)=o(1)}
\Theta (x;\lambda )=o(1) \quad \text{for} \quad \lambda \in
S_{\pm},\ \lambda \to \infty \qquad (\text{for every}\  x \in
[0,1]).
\end{equation}
Further, denote
\begin{equation*}
G_f(x;\lambda ):=\overline{(R^*_L(\lambda)f)(x)}.
\end{equation*}
By~\eqref{2.45} and~\eqref{2.47}--\eqref{2.49}, for $\lambda \in
S_{\pm}$, the $G_f(x;\lambda )$ admits the representation
 \begin{multline}\label{2.51b}
\overline{G_f(x;\lambda )}=(iB^{-1})^*\int ^1_x\mathcal{I}_n^{-1}(x;\lambda)^*
 \widetilde{e}(x-t;\lambda)\mathcal{I}_n^*(t;\lambda) f(t)dt= \\
 =(iB^{-1})^*\mathcal{I}_n^{-1}(x;\lambda)^*
 \Theta(x;\lambda),
 \end{multline}
and hence from~\eqref{Theta(x)=o(1)} we conclude that
\begin{equation}\label{Gfx=o(1)}
G_f(x;\lambda)=o(1) \ \ \text{for} \ \ \lambda \in S_{\pm},\ \lambda \to
\infty.
\end{equation}
But $G_f(x;\lambda )$ is the entire function of exponential type (for every
$x\in [0,1]$). Moreover, since $G_f(x;\lambda )$ is bounded along the pair of
rays in $S_+$ and along the pair of opposite rays in $S_-$, it is bounded in
$\mathbb{C}$ due to the Phragm\'{e}n-Lindel\"{o}f theorem~\cite{Lev56}. By the
Liouville theorem, $G_f(x;\lambda )$ does not depend on $\lambda$, i.e.,
$G_f(x;\lambda )=:c(x), \ x\in [0,1]$. Due to~\eqref{Gfx=o(1)} the function
$c(x)$ is zero and hence $(R_L^*(\lambda)f)(x)=G_f(x;\lambda )\equiv 0$. It
follows that $f = 0$.

(vii) The minimality of the system of EAF follows from
Lemma~\ref{Lem_minimality} applied  to the resolvent operator
$R_{L_{C,D}}(\lambda)$ with $\lambda \in \rho(L_{C,D})$.
    \end{proof}
  \begin{corollary} \label{cor2.1}
Let  $Q\in L^2[0,1]\otimes {\Bbb C}^{n\times n}$ and let the
matrices $T_{zB}(C,D)$ and $T_{-zB}(C,D) = T_{zB}(D,C)$ be
nonsingular for some $z\in \mathbb{C}.$  Then

(i) The boundary conditions \eqref{1.3} are weakly $B$-regular.

(ii) The system of EAF of the operator $L_{C,D}(Q)$ is complete
and minimal in $L^2\left([0,1];{\Bbb C}^n\right)$.
  \end{corollary}
   \begin{proof}[Proof]
Since all the numbers $\Re (zb_k)$ are different from zero, we get
that, for sufficiently small $\delta$, the signs  of $\Re
((1\pm\delta)z b_k)$ coincide with the sign of $\Re (zb_k)$. It
follows that the matrices $T_{zB}(C,D)$, $T_{(1+\delta)zB}(C,D)$
and $T_{(1-\delta)zB}(C,D)$ coincide. Thus, we can apply
Theorem~\ref{th2.1} to the operator $L_{C,D}(Q)$ and  the points
$z_1=(1+\delta)z$, $z_2=(1-\delta)z$ and $z_3=-z$.
       \end{proof}

\subsection{Completeness result for adjoint operator}

   \begin{corollary}\label{cor2.1A}
Let boundary conditions~\eqref{1.3} be weakly $B$-regular. Then

(i)\ The boundary conditions
   \begin{equation}\label{3.58}
C_*g(0) + D_*g(1)=0
   \end{equation}
of the adjoint boundary value problem are weakly $B^*$-regular.

(ii)\  The system of root functions of the adjoint operator
$L^*_{C,D}$ is complete and minimal in $L^2[0,1]\otimes
\mathbb{C}^n$.
     \end{corollary}
     \begin{proof}[Proof]
(i)\   The adjoint operator $L^*_{C,D} := (L_{C,D})^*$ is defined
as a restriction of the maximal  differential operator
   \begin{equation*} L^*:=\frac{1}{i}B^*\otimes
\frac{d}{dx}+Q^*(x), \qquad \dom(L^*) =  W^1_2([0,1];{\Bbb C}^n),
    \end{equation*}
to the domain
$\dom(L^*_{C,D}) =  \{g\in W^1_2([0,1];{\Bbb C}^n):\ C_*g(0) +
D_*g(1)=0\}.$
Moreover, if $Cf(0) + Df(1)=0$ and $C_*g(0) + D_*g(1)=0$, we have
     \begin{equation}\label{3.t3.m.2}
\langle Bf(0), g(0)\rangle - \langle Bf(1), g(1)\rangle = 0.
      \end{equation}
Put $\widetilde{B} :=\diag (B, -B)$ and consider  ${\mathcal H}
=\mathbb{C}^n\oplus \mathbb{C}^n$ as a space with bilinear  form
   \begin{equation}\label{3.59A}
w (\varphi,\psi) := \langle \widetilde{B}{\varphi},\psi \rangle =
\left\langle B{\varphi}_1,\psi_1 \right\rangle - \left\langle
B{\varphi}_2, \psi_2 \right\rangle,
    \end{equation}
where  ${\varphi}=\col({\varphi}_1,{\varphi}_2),\ \psi =\col(\psi
_1,\psi _2).$
The condition~\eqref{3.t3.m.2} means that the subspace
$\Ker(C_*\,\,D_*)$ is the right  $w$-orthogonal to  $\Ker(C\,\,D)$ in
${\mathcal H}$.  Since $\dim \Ker(C\,\,D)=\dim \Ker(C_*\,\,D_*)=n$,  the
subspace $\Ker(C\,\,D)$ is non-degenerate and $\{\Ker(C\,\,D)\}^{\perp}
= \Ker(C_*\,\,D_*)$, i.e. $\Ker(C_*\,\,D_*)$ is the (right)
$w$-orthogonal complement of $\Ker(C\,\,D).$

Let $\beta_1,\beta_2,\dots, \beta_{2n}$ be the eigenvalues of
$\widetilde B$ and let $e_1, e_2,\dots, e_{2n}$ be the
corresponding  eigenvectors. For every admissible  $z$ (i.e. such
that $z\beta_k \not \in i\Bbb R$ for every $k\le 2n$) we put
${\mathcal H}_z = \Span\{e_k:\ \Re(z\beta_k)>0 \}.$ Since
$\beta_{n+k} = - \beta_{k}\in \sigma(\widetilde B),\ k\in
\{1,\ldots, n\}$, $\dim{\mathcal H}_z=n$ for every admissible $z$.

Next we note that
       \begin{equation}\label{3.59}
T_{zB}(C,D)= \left. (C\,\,D) \right|_{{\mathcal H}_z}.
       \end{equation}
Therefore,  $\det T_{zB}(C,D)\ne 0$ if and only if $ \left. \Ker((C\,\,D)
\right|_{{\mathcal H}_z})=\{0\},$ i.e.
$\Ker(C\,\,D)\cap {\mathcal H}_z = \{0\}$.  Since  $\dim
\Ker(C\,\,D)=\dim {\mathcal H}_z=n$, the latter identity is also
valid for the  right $w$-orthogonal complements of these
subspaces, i.e. $\Ker(C_*\,\,D_*)\cap {\mathcal H}_{-z} = \{0\}.$

Alongside the space ${\mathcal H}$, we  consider the same space
${\mathcal H_*}=\mathbb{C}^{2n}=\mathbb{C}^n\oplus \mathbb{C}^n$
equipped with another non-degenerate  bilinear form
   \begin{equation*}
w^*({\varphi},\psi) := \langle
\widetilde{B}^*{\varphi},\psi\rangle = \left\langle
B^*{\varphi}_1,\psi_1 \right\rangle -\left\langle B^*{\varphi}_2,
\psi_2 \right\rangle.
    \end{equation*}
Next we  define the corresponding subspaces ${\mathcal H_*}_z$
with respect to the form $w^*(\cdot,\cdot)$ (matrices
$z\widetilde{B}^*$) and note that
       \begin{equation}\label{3.60A}
T_{zB^*}(C_*,D_*)= \left. (C_*\,\,D_*) \right|_{{\mathcal H_*}_z}.
       \end{equation}

Since $\Re(\overline{z \beta_k})=\Re(z \beta_k)$,  one has
${\mathcal H_*}_z={\mathcal H}_{\overline{z}}$. Hence
$\Ker(C\,\,D)\cap {\mathcal H}_z = \{0\}$ is equivalent to
$\Ker(C_*\,\,D_*)\cap {\mathcal H}_{*-\overline{z}} = \{0\}.$
Combining this equivalence with  relations \eqref{3.59} and
\eqref{3.60A}  we get
$$
\det T_{zB}(C,D) \ne 0 \ \Longleftrightarrow \ \det
T_{-\overline{z}B^*}(C_*,D_*) \ne 0.
$$
Hence boundary conditions \eqref{3.58} are weakly  $B^*$-regular
and conditions of Definition \ref{def1.1} are satisfied with
points $-\overline{z_1},-\overline{z_2},-\overline{z_3}$.

(ii) Combining statement (i) with Theorem~\ref{th2.1} we get the
result.
  \end{proof}
         \begin{remark} \label{rem2.1}
(i)\  Theorem \ref{th2.1} remains valid for the
integro-differential operator
\begin{equation} \label{2.3a}
-iBy'+ Q(x)y + \int_0^xM(x,t)y(t)\,dt=\lambda y, \qquad y\in\col
(y_1,y_2,\dots,y_n),
\end{equation}
with a kernel $M(x,t)\in L^\infty(\Omega)\otimes \Bbb C^{n\times
n}$.

(ii) If the maximality condition \eqref{1.4}  is violated, i.e.
${\rm rank}(C \ \ D) \le  n-1$, then  the characteristic
determinant \eqref{2.11} is identical zero.  Indeed, in this case
  $$
{\rm rank}(C + D\Phi(1;\lambda)) = {\rm rank}\left((C\ \
D)\binom{I_n}{\Phi(1;\lambda)}\right)\le  {\rm rank}(C\ \ D) \le
n-1.
  $$
Hence $\Delta_\Phi(\lambda) = \det(C + D\Phi(1;\lambda)) \equiv 0,
\quad \lambda\in \Bbb C.$

Note however that the latter  might happen even whenever ${\rm
rank}(C \ \ D) = n.$
        \end{remark}

\subsection{Examples}

\begin{example} \label{ex1}
Assume that $C\in {\Bbb C}^{n\times n},$ and  $\det C\not = 0.$
Let also  $D=CM$, where $M\in {\Bbb C}^{n\times n}$ and all its
principal minors  are nonsingular. In this case, the matrix
$T_A(I_n,M)$ is nonsingular for every matrix $A$. Hence the matrix
$T_A(C,D)=CT_A(I,M)$ is always nonsingular.

For instance, the boundary conditions
       \begin{equation}\label{3.60}
 y_j(0)=d_jy_j(1), \quad d_j \ne 0, \quad j\in\{1,2,\dots,n\},
       \end{equation}
that include the periodic ones   $(d_j = 1)$ have this form with
$C=I_n$ and $D=\diag(d_1,d_2,\ldots,d_n)$ and hence are weakly
$B$-regular for any non-singular $B$.

          \end{example}

Note that  conditions \eqref{3.60}  are regular, i.e., the matrix
$T_{zB}(C,D)$ is nonsingular  for every admissible  $z\in \Bbb C.$

Next we present several examples of \emph{irregular BC \eqref{1.3}
that are weakly $B$-regular}.  To this end we prove  the following
fact mentioned in the Introduction.
  \begin{lemma}\label{lem_reg}
Assume that the boundary conditions~\eqref{1.3} split in  $k$
conditions at 0 and $n-k$ conditions at $1$. Then

(i) If  $\Re(zB)$ is  invertible  and $\det T_{zB}(C,D) \ne 0,$
then $k=\kappa_+(\Re(zB))$.

(ii) If the boundary conditions are regular,  then $n=2k$ and
$\kappa_+(\Re(zB))=\kappa_-(\Re(zB))$ for every admissible $z\in
\Bbb C$, i.e., for those $z$ that  $\Re(zB)$ is invertible.
  \end{lemma}
\begin{proof}[Proof]
(i) Let  $z \in \mathbb{C}$ be admissible, i.e. the matrix
$\Re(zB)$ is nonsingular. Then  the matrix $T_{zB}(C,D)$ exists
and  has $l$ columns from $C$ and $n-l$ columns from $D$. By the
definition, $l=\kappa_+(\Re(zB))$. Further,  since the last $n-k$
rows of the matrix $C$ and the first $k$ rows of the  matrix $D$
are zero, the matrix $T_{zB}(C,D)$ has at least two zero
submatrices of sizes $(n-k) \times l$ and $k \times (n-l)$. Since
$\det T_{zB}(C,D)\not = 0$, one has $n-k+l \le n$ and $k+n-l \le
n$. Hence $k=l$.

(ii) Let the boundary conditions  be regular and $\det\Re(zB)\not
= 0.$  Then both  matrices $T_{zB}(C,D)$ and $T_{-zB}(D,C)$ are
well-defined  and nonsingular. By the statement (i),
$k=\kappa_+(\Re(zB))$ and $k=\kappa_+(\Re(-zB))$. Since
$\kappa_+(\Re(-zB))=\kappa_-(\Re(zB))$, one has
$\kappa_+(\Re(zB))=\kappa_-(\Re(zB))$ and
$2k=\kappa_+(\Re(zB))+\kappa_-(\Re(zB))=n$.
   \end{proof}
    \begin{example}\label{ex.k.k+1}
Let $n=2k+1$, $B=\diag(b_1, \ldots, b_n)$ with
$b_j=\exp\left(\frac{2\pi i j}{n}\right),\ j\in\{1,2,\ldots,n\},$
and let BC~\eqref{1.3} split in $k$ conditions at 0 and $k+1$
conditions at $1$. Then the lines $\{z\in\mathbb{C}:\Re(i z
b_j)=0\}$ divides $\mathbb{C}$ in $2n$ sectors
$\sigma_1,\sigma_2,\ldots,\sigma_{2n}$ such that the point $iz_p$
belongs to the interior of $\ \sigma_p,\ p\in\{1,2,\ldots,2n\}$,
where $z_p=\exp\left(\frac{\pi i p}{n}\right)$. Note that for $p
\equiv k \pmod 2$ we have $\kappa_+(\Re(z_p B))=k+1$ and hence, by
Lemma~\ref{lem_reg}, the matrix $T_{z_p B}(C,D)$ is singular.

However, in general, for other values of $p$ the matrix $T_{z_p
B}(C,D)$ is nonsingular. More precisely, if $p \equiv k+1 \pmod 2$
then $\kappa_+(\Re(z_p B))=k$ and
\begin{equation}\label{TzBCD=CD}
\det T_{z_p B}(C,D) = C\left(\begin{array}{cccc}
  1 & 2 & \ldots & k \\
  j_1 & j_2 & \ldots & j_{k} \\
\end{array}\right)
\cdot D\left(\begin{array}{ccc}
  k+1  & \ldots & n \\
  j_{k+1} & \ldots & j_{n} \\
\end{array}\right)
\end{equation}
where $1 \le j_1<j_2<\ldots<j_k \le n$, $1 \le
j_{k+1}<j_{k+2}<\ldots<j_n \le n$, $\Re(z_p b_{j_{\nu}})>0$ for $1
\le \nu \le k$ and $\Re(z_p b_{j_{\nu}})<0$ for $k+1 \le \nu \le
n$. Here $ A\left(\begin{array}{cccc}
  j_1 & j_2 & \ldots & j_p \\
  k_1 & k_2 & \ldots & k_p \\
\end{array}\right)
$ stands for the  minor of $n \times m$-matrix $A = (a_{jk})$
composed of the entries in the  rows with the indices
$j_1,\ldots,j_p \in \{1,\ldots,n\}$ and the columns with the
indices $k_1,\ldots,k_p \in \{1,\ldots,m\}$.

Assume that for some values $p_1,p_2,p_3 \in \{1,2,\ldots,2n\}$
satisfying  $p_1<p_2<p_3$, $p_1 \equiv p_2 \equiv p_3 \equiv k+1
\pmod 2$, $p_2-p_1<n$, $p_3-p_2<n$ and $p_3-p_1>n$ the
corresponding minors of matrices $C$ and $D$ from
equality~\eqref{TzBCD=CD} for values $p=p_1,p_2,p_3$ are non-zero.
Then the boundary conditions~\eqref{1.3} will be weakly
$B$-regular if we put  $z_j=\exp\left(\frac{\pi i p_j}{n}\right),\
j\in \{1,2,3\},$ in Definition \ref{def1.1} of weak
$B$-regularity. However, by Lemma~\ref{lem_reg},  these boundary
conditions are irregular.

One obtains an explicit example by setting $n=3$ and
\begin{equation*}
   \begin{cases}
 c_{11}y_1(0)  + c_{12}y_2(0) + c_{13}y_3(0)&=0 \\
 d_{21}y_1(1)  + d_{23}y_3(1)&=0 \\
 d_{32}y_2(1)  + d_{33}y_3(1)&=0
   \end{cases}\ \
 \end{equation*}
where all the coefficients are non-zero. Here we can take $p_j =
2j$, $j\in \{1,2,3\}.$

We obtain another explicit example of irregular but weakly
$B$-regular splitting  boundary conditions \eqref{1.3} for  system
\eqref{1.1} with $n=2k+1$, by setting
$$
(C \ D)=\left(%
\begin{array}{cccccccc}
  1         & 1         & \ldots & 1         & 0      & 0      & \ldots & 0      \\
  c_1       & c_2       & \ldots & c_n       & 0      & 0      & \ldots & 0      \\
  \vdots    & \vdots    & \ddots & \vdots    & \vdots & \vdots & \ddots & \vdots \\
  c_1^{k-1} & c_2^{k-1} & \ldots & c_n^{k-1} & 0      & 0      & \ldots & 0      \\
  0         & 0         & \ldots & 0         & 1      & 1      & \ldots & 1      \\
  0         & 0         & \ldots & 0         & d_1    & d_2    & \ldots & d_n    \\
  \vdots    & \vdots    & \ddots & \vdots    & \vdots & \vdots & \ddots & \vdots \\
  0         & 0         & \ldots & 0         & d_1^k  & d_2^k  & \ldots & d_n^k  \\
\end{array}%
\right).
$$
Here  $c_j\not = c_k$ and  $d_j\not = d_k$ for $j\not = k.$ Now
any $k \times k$-minor of the matrix $C$ that corresponds to its
first $k$ rows is the Vandermonde determinant,  hence it is
non-zero. The same  is true for any $(k+1) \times (k+1)$-minor of
the matrix $D$ that corresponds to its last $k+1$ rows. Hence
$\det T_{z_pB}(C,D) \ne 0$ for any $p \in \{1,2,\ldots,2n\}$ such
that $p \equiv k+1 \pmod 2$. So, we can take $p_1=2, p_2=4,
p_3=n+3$ for odd $k$  and  $p_1=1, p_2=3, p_3=n+2$ for even $k$.
 \end{example}

Next we present  two examples of \emph{non-splitting} boundary
conditions that are \emph{irregular but weakly $B$-regular}.
 \begin{example}
 Let $n=3$, $B=\diag(b_1, b_2, b_3)$ and $b_j=\exp\left(\frac{2\pi i j}{3}\right),\ j\in\{1,2,3\}$.
Consider the boundary conditions \eqref{1.3} of the form:
      \begin{equation*}
      \begin{cases}
 y_1(0)=d_{12}y_2(1)+d_{13}y_3(1) \\
 y_2(0)=d_{21}y_1(1)+d_{23}y_3(1) \\
 y_3(0)=d_{31}y_1(1)+d_{32}y_2(1)
       \end{cases},
       \end{equation*}
where all the coefficients $d_{jk}$  are non-zero. In this case,
the matrix $T_{zB}(C,D)$ is nonsingular for
$z=\exp\left(\frac{2\pi i j}{3}\right),\ j\in\{1,2,3\},$ but it is
singular for $z=-\exp\left(\frac{2\pi i j}{3}\right),\
j\in\{1,2,3\}$. For instance, for $z=1$ we have
$$
\det T_{zB}(C,D)=\det T_{B}(C,D)=\det \left(%
\begin{array}{ccc}
  0      & d_{12} & 0 \\
  d_{21} & 0      & 0 \\
  d_{31} & d_{32} & 1 \\
\end{array}%
\right) = -d_{12}d_{21} \ne 0.
$$
At the same time,  for $z=-1$ one has
$$
\det T_{zB}(C,D)=\det T_{-B}(C,D)=\det \left(%
\begin{array}{ccc}
  1 & 0 & d_{13} \\
  0 & 1 & d_{23} \\
  0 & 0 & 0 \\
\end{array}%
\right) = 0.
$$
 \end{example}

\begin{example}\label{ex.3.3}
Let $n=3$, $B=\diag(b_1, b_2, b_3)$ and $b_j=\exp\left(\frac{2\pi
i j}{3}\right),\ j\in\{1,2,3\}$. Consider boundary conditions
\eqref{1.3} of the form:
\begin{equation*}
c_1 y_1(0)=c_2 y_2(0)=c_3 y_3(0)=d_1 y_1(1)+d_2 y_2(1)+d_3 y_3(1),
\end{equation*}
where all the coefficients are non-zero. In this case, the matrix
$T_{zB}(C,D)$ is nonsingular for $z=-\exp\left(\frac{2\pi i
j}{3}\right),\ j\in\{1,2,3\},$ but it is singular for
$z=\exp\left(\frac{2\pi i j}{3}\right),\ j\in\{1,2,3\}$. For
instance, for $z=-1$ we have
$$
\det T_{zB}(C,D)=\det T_{-B}(C,D)=\det \left(%
\begin{array}{ccc}
  c_1 & 0   & d_3 \\
  0   & c_2 & d_3 \\
  0   & 0   & d_3 \\
\end{array}%
\right) = c_1 c_2 d_3 \ne 0.
$$
On the other hand,  for $z=1$
$$
\det T_{zB}(C,D)=\det T_{B}(C,D)=\det \left(%
\begin{array}{ccc}
  d_1 & d_2 & 0 \\
  d_1 & d_2 & 0 \\
  d_1 & d_2 & c_3 \\
\end{array}%
\right) = 0.
$$
 \end{example}

\section{The case of a selfadjoint matrix $B=B^*$}

Suppose that $B = B^*\in {\Bbb C}^{n\times n}$ and $\det B\not =
0.$ To state the next result, we denote by $P_+$ and $P_-$ the
spectral projectors onto "positive"\ and "negative"\ parts of the
spectrum of a selfadjoint matrix $B=B^*$, respectively, and put
       \begin{equation}\label{2.2T+}
T_{\pm} := T_{\pm}(B;C,D):= CP_{\pm}  + DP_{\mp}.
       \end{equation}
%
%
        \begin{proposition} \label{th4}
Assume  that $B=B^*$ and  $Q\in L^2[0,1]\otimes {\Bbb C}^{n\times
n}$. If
\begin{equation}\label{2.4}
\det T_+(B;C,D)\ne 0 \qquad\text{and}\qquad \det T_-(B;C,D)\ne 0,
\end{equation}
then the system of EAF of the operator $L_{C,D}$ is complete and
minimal in the space $L^2[0,1] \otimes \mathbb{C}^n$.
 \end{proposition}
 \begin{proof}[Proof]
To prove the completeness, it suffices to note that
\begin{equation*}
T_+(B;C,D)=T_{B}(C,D) \qquad\text{and}\qquad T_-(B;C,D)=T_{B}(D,C)
\end{equation*}
and to put $z=1$ in Corollary~\ref{cor2.1}.
 \end{proof}

Next we clarify  Proposition \ref{th4} for accumulative
(dissipative) BVP.
Recall  that an operator $T$ in a Hilbert space $\frak H$ is
called accumulative (dissipative) whenever
  $$
\im( Tf, f)\le 0\ (\ge 0), \qquad f\in \dom(T).
  $$
%
%

   \begin{lemma}\label{lem_dissip}
Let $B=B^*$ and let the operator $L_{C,D}(0)$ be accumulative
(dissipative). Then $\det T_+(B;C,D)\ne 0$ $(\det T_-(B;C,D)\ne
0).$
\end{lemma}
     \begin{proof}[Proof]
Since the operator $L_{C,D}(0)$ is accumulative, one has
  \begin{equation}\label{4.2A}
2\im(L_{C,D}(0)y,y) = \langle B y(0),y(0) \rangle - \langle B
y(1),y(1)\rangle \le 0, \quad y \in \dom(L_{C,D}).
     \end{equation}
As in the proof of Corollary~\ref{cor2.1A}  we let $\widetilde{B}
:=\diag (B, -B) = \widetilde{B}^*$ and equip  the space ${\mathcal
H} =\mathbb{C}^n\oplus \mathbb{C}^n$  with the non-degenerate
Hermitian bilinear form \eqref{3.59A}.  Let also
$\beta_1,\beta_2,\dots, \beta_{2n}$ be the eigenvalues of
$\widetilde B$ and let $e_1, e_2,\dots, e_{2n}$ be the
corresponding  eigenvectors. We put ${\mathcal H}_{\pm} :=
\Span\{e_k:\ \pm\Re(\beta_k)>0 \}$ and note that $\dim{\mathcal
H}_{\pm}=n.$

Further, for any $y(\cdot) \in \dom(L_{C,D})$ the vector
$\varphi=\col(y(0),y(1))(\in \Ker(C\ D))$ is non-positive in
$\mathcal{H}$, i.e., $ \langle \widetilde{B} \varphi,\varphi
\rangle \le 0$. Hence $\Ker(C \ D) \subset \{\psi \in \mathcal{H}:
\langle \widetilde{B} \psi,\psi \rangle \le 0 \}$. On the other
hand, $\langle \widetilde{B} \psi,\psi \rangle \ge 0$ for any
$\psi \in \mathcal{H}_+.$  Hence $\Ker(C \ D) \cap \mathcal{H}_+ =
\{0\}$ and due to \eqref{3.59} $\det T_+(B;C,D)\ne 0$.
\end{proof}

   \begin{corollary}\label{cor2.2}
Assume that the operator $L_{C,D}(0)$ is accumulative
(dissipative), and $\det T_-(B;C,D)\ne 0$ $(\det T_+(B;C,D)\ne
0).$  Then both conditions \eqref{2.4} are satisfied and the
system of root functions of the operator $L_{C, D}(Q)$ with $Q\in
L^2[0,1]\otimes {\Bbb C}^{n\times n}$ is complete and minimal in
$L^2[0,1]\otimes\Bbb C^n$.
      \end{corollary}
       \begin{proof}[Proof]
Combining  Lemma~\ref{lem_dissip} with  Proposition~\ref{th4}
yields the statement.
   \end{proof}
    \begin{corollary}\label{cor2.3}
Suppose that $C, D\in \Bbb C^{n \times n}$ satisfy both the
maximality condition \eqref{1.4} and the relation $CB^{-1}C^* -
DB^{-1} D^* = 0$. Then the system of root functions of the
operator $L_{C, D}(Q)$ is complete and minimal in $L^2[0,1]\otimes
\Bbb C^n$.
    \end{corollary}
   \begin{proof}[Proof]
It follows from the assumptions of the corollary that the operator
$L_{C, D}(0)$ with $Q=0$ is selfadjoint. It remains to apply
Corollary \ref{cor2.2}.
    \end{proof}
      \begin{remark}\label{rem3.3}
$(i)$ In case of the $2\times 2$ Dirac system, a close problem on
completeness of matrix solutions in the space of matrix functions
is studied in~\cite{Mar77}. Moreover, conditions~\eqref{2.4} are
equivalent to conditions  (1.3.39) from~\cite{Mar77}.

$(ii)$ In case of the simplest operator $L_{C, D} = -iI_n\otimes
\frac{d}{dx}\ (B=I_n, Q=0)$,  another (and rather complicated)
proof of Corollary~\ref{cor2.2} was obtained in~\cite{Gin71}.

$(iii)$  Corollary~\ref{cor2.3} is implied by the known
M.V.~Keldysh theorem (\cite{Kel51}, \cite{GohKre65},
\cite{Markus86})  since the operator $L_{C, D}(0)$ of the
form~\eqref{1.1}, \eqref{1.3} with $Q=0$ is selfadjoint, and its
resolvent has a finite (equal to 1) order.
    \end{remark}

Next  we show that, in the case of  zero potential matrix,
$Q\equiv0$,  conditions~\eqref{2.4} of Proposition \ref{th4} are
also necessary.
\begin{proposition} \label{prop.neces}
The  system of root functions of the boundary problem
\begin{gather} \label{5.t5.1}
-iBy' = \lambda y, \quad B=B^*, \qquad y=\col(y_1,...,y_n), \\
 \label{5.t5.2} Cy(0) + Dy(1)=0,
  \end{gather}
is incomplete in $L^2[0,1] \otimes \mathbb{C}^n$ whenever $\det
T_-(B;C,D)=0.$ Moreover, in this case its defect is infinite.
        \end{proposition}
\begin{proof}[Proof]

Since $\det T_-(B;C,D)=0$,  one of the boundary conditions is of
the form
\begin{equation}\label{5.t5.4}
\sum_{k=1}^n c_ky_k(\xi;\lambda)=0,\qquad\text{where}\quad
\begin{cases}
\xi=0,\quad\text{for}\ b_k>0, \\
\xi=1,\quad\text{for}\ b_k<0.
\end{cases}
\end{equation}
Every solution $Y(x;\lambda)=\col(y_1(x;\lambda),\dots,
y_n(x;\lambda))$ of equation~\eqref{5.t5.1} satisfying
condition~\eqref{5.t5.4} admits the following representation
\begin{equation}\label{4.13}
y_k(x;\lambda)=
 \begin{cases}
  a_ke^{i |b_k|\lambda x},\quad\text{for}\ b_k>0, \\
  a_ke^{i |b_k|\lambda (1-x)},\quad\text{for}\
  b_k<0,
 \end{cases}
 \qquad\text{where} \quad  \sum_{k=1}^n c_ka_k=0.
\end{equation}
Let $\alpha$ be a positive number such that
$\tfrac{\alpha}{|b_k|}<1$ for all $k$. We put for $b_k>0$
\begin{equation*}
 \varphi_k(x)=
    \begin{cases}
      \overline{c_k}|b_k|,\quad\text{for}\ 0 \le x<\tfrac{\alpha}{|b_k|} \\
      0,\quad\text{for}\ \tfrac{\alpha}{|b_k|}\le x\le 1
    \end{cases},
\end{equation*}
for $b_k<0$
\begin{equation*}
 \varphi_k(x)=
    \begin{cases}
      \overline{c_k}|b_k|,\quad\text{for}\
      0\le 1-x < \tfrac{\alpha}{|b_k|} \\
      0,\quad\text{for}\ \tfrac{\alpha}{|b_k|} \le 1-x \le 1
    \end{cases}
\end{equation*}
and $\Phi(x) :=\col(\varphi_1(x)),\dots, \varphi_n(x)).$
From \eqref{4.13} one gets
\begin{multline*}
( y_k(x;\lambda),\ \varphi_k(x))_{L^2[0,1]}
 =\int_0^{\frac{\alpha}{|b_k|}} a_ke^{i\lambda |b_k|x} c_k|b_k|\, dx=
 c_ka_k \int_0^\alpha e^{i\lambda t}\, dt.
\end{multline*}
Here we use the change  $x\to 1-x$ for $b_k<0$.  It follows that
\begin{multline*}
\left( Y(x;\lambda),\ \Phi(x) \right)=
 \sum_{k=1}^n \left( y_k(x;\lambda),\ \varphi_k(x) \right)_{L^2[0,1]}
 =\left(\sum_{k=1}^n c_ka_k\right) \int_0^\alpha e^{i\lambda t}\, dt=0.
\end{multline*}
Thus, $\Phi(\cdot)$ is orthogonal to all the solutions of
equation~\eqref{5.t5.1} satisfying  condition  ~\eqref{5.t5.4}.
Hence it is orthogonal to the system of root functions of the
operator $L_{C,D}$. Thus,  the system of root functions of the
operator $L_{C,D}$ is incomplete.
 \end{proof}
\begin{remark}
All the results of this section including Theorem~\ref{th2.1} and
Propositions \ref{th4} and \ref{prop.neces} remain valid  (with
the same  proofs) for $Q\in L^1[0,1] \otimes \mathbb{C}^{n\times
n}.$ We stated  them for $Q\in L^2[0,1] \otimes
\mathbb{C}^{n\times n}$ because only in this case the domain
$\dom(L_{C,D}(Q))$ has simple description \eqref{1.3BB}. Moreover,
the results on completeness remain valid for the spaces $L^p[0,1]
\otimes \mathbb{C}^n$ with $p \in [1,\infty).$
\end{remark}

\section{Irregular BVP for $2\times 2$ Dirac type systems}

\subsection{Sufficient conditions of completeness}

Here we substantially supplement Proposition ~\ref{th4} confining
ourselves to the case of the second order system $(n=2)$. We
consider irregular BC and  indicate other completeness conditions
that \emph{depend} on $Q$. In particular, we show that, as
distinct from the case $Q(\cdot)\equiv0,$ conditions~\eqref{2.4}
of Proposition~\ref{th4} are not necessary for the completeness of
the system of root functions even in the case of
$Q(\cdot)=Q^*(\cdot)\not\equiv0$ and dissipative (accumulative)
boundary conditions.

Consider the $2\times 2$ Dirac type system:
    \begin{equation}\label{3.1}
-i B y'+Q(x)y=\lambda y, \qquad y=\col(y_1,y_2), \qquad x\in[0,1],
    \end{equation}
where
    \begin{equation}\label{3.2}
B=\diag(b^{-1}_1,b^{-1}_2), \quad b_1< 0< b_2\quad \text{and}\quad
Q=
\begin{pmatrix}
0&Q_{12}\\
Q_{21}&0
\end{pmatrix}.
         \end{equation}
To the system~\eqref{3.1}  we join  boundary
conditions~\eqref{1.3} rewritten for convenience in the form
    \begin{equation}\label{3.3}
U_j(y) := a_{j 1}y_1(0) + a_{j 2}y_2(0) + a_{j 3}y_1(1) + a_{j
4}y_2(1)= 0, \quad  j\in\{1,2\}.
    \end{equation}
Further, let $\Phi(x;\lambda)$ be the fundamental matrix of the
system~\eqref{3.1} (uniquely) determined  by the initial condition
$\Phi(0;\lambda)=I_2$, i.e.,
    \begin{equation*}
\Phi(x;\lambda) :=
\begin{pmatrix}
\Phi_1(x;\lambda)&  \Phi_2(x;\lambda)\end{pmatrix}, \quad
\Phi_j(x;\lambda) :=
\begin{pmatrix}
\varphi_{1j}(x;\lambda) \\
\varphi_{2j}(x;\lambda)
\end{pmatrix},
\quad j\in\{1,2\},
    \end{equation*}
where   $\displaystyle \Phi_1(0;\lambda) := \binom{1}{0},\
\Phi_2(0;\lambda) = \binom{0}{1}$. The eigenvalues of
problem~\eqref{3.1}--\eqref{3.3} are the roots of the
characteristic equation $\Delta(\lambda) := \det U(\lambda)=0$,
where
    \begin{equation}\label{3.3A}
U(\lambda):=
    \begin{pmatrix}
U_1(\Phi_1(x;\lambda))& U_1(\Phi_2(x;\lambda)) \\
U_2(\Phi_1(x;\lambda))& U_2(\Phi_2(x;\lambda))
    \end{pmatrix} =:
        \begin{pmatrix}
u_{11}(\lambda)& u_{12}(\lambda) \\
u_{21}(\lambda)& u_{22}(\lambda)
    \end{pmatrix}.
               \end{equation}
By putting $J_{jk}=\det
\begin{pmatrix}
a_{1j}&a_{1k}\\
a_{2j}&a_{2k}
\end{pmatrix},
\  j,k\in\{1,\ldots,4\}$,  we arrive at the following expression
for the characteristic determinant:
   \begin{equation}\label{3.5}
\Delta(\lambda) = J_{12} + J_{34}e^{i(b_1+b_2)\lambda} +
J_{32}\varphi_{11}(\lambda) + J_{13}\varphi_{12}(\lambda) +
J_{42}\varphi_{21}(\lambda) + J_{14}\varphi_{22}(\lambda),
   \end{equation}
where $\varphi_{jk}(\lambda) := \varphi_{jk}(1;\lambda)$. If $Q=0$
then $\varphi_{12}(x;\lambda) = \varphi_{21}(x;\lambda) = 0$, and
the characteristic determinant $\Delta_0(\cdot)$ has the form
   \begin{equation}\label{3.5A}
\Delta_0(\lambda) = J_{12} + J_{34}e^{i(b_1+b_2)\lambda} +
J_{32}e^{ib_1\lambda} + J_{14}e^{ib_2\lambda}.
   \end{equation}
For the problem~\eqref{3.1}--\eqref{3.2} we have $\det(T_+) =
J_{32}$ and  $\det(T_-) = J_{14}$ where $T_{\pm}$ are defined by
\eqref{2.2T+}.  Thus, condition~\eqref{2.4} means that
$J_{32}\cdot J_{14} = \det(T_1\cdot T_2)\not =0$ and presents
\emph{the regularity condition of
problem}~\eqref{3.1}--\eqref{3.3}. For the Dirac system $(-b_1 =
b_2=1)$, \emph{the regularity condition} is stronger than
\emph{the nondegeneracy} of boundary conditions; the last one
means that $\Delta_0(\lambda) \not = J_{12} + J_{34}= \const$.

  \begin{theorem}\label{th3.1}
Let $Q_{12}(\cdot), Q_{21}(\cdot) \in C[0,1]$.  If
\begin{gather}
  \label{3.7A} \left|J_{32}\right|+\left|b_1 J_{13}Q_{12}(0)+b_2 J_{42}Q_{21}(1)\right| \ne 0, \\
  \label{3.7B} \left|J_{14}\right|+\left|b_1 J_{13}Q_{12}(1)+b_2 J_{42}Q_{21}(0)\right| \ne 0,
 \end{gather}
then the system of root functions of
problem~\eqref{3.1}--\eqref{3.3} (i.e. of the operator
$L_{C,D}(Q)$) is complete and minimal in $L^2\left([0,1];{\Bbb
C}^2\right)$.
   \end{theorem}
      \begin{corollary}\label{cor5.2}
Let $Q_{12}(\cdot), Q_{21}(\cdot) \in C[0,1]$,  and let
$J_{32}=J_{14}=0$.  If
   \begin{gather}
  \label{3.7AA}  b_1 J_{13}Q_{12}(0)+b_2 J_{42}Q_{21}(1) \ne 0, \\
  \label{3.7BB}  b_1 J_{13}Q_{12}(1)+b_2 J_{42}Q_{21}(0) \ne 0,
 \end{gather}
then the system of root functions of
problem~\eqref{3.1}--\eqref{3.3} is complete and minimal.
         \end{corollary}
   \begin{remark}
(i) In the case $-b_1=b_2=1$, Theorem~\ref{th3.1} gives
completeness even in the case of degenerated boundary conditions.

(ii) If $J_{32}=J_{14}=0$, $Q_{12}(\cdot)=Q_{21}(\cdot)$ and
$Q_{12}(0)=Q_{12}(1)\not=0$, then
conditions~\eqref{3.7A}--\eqref{3.7B} acquire the form $b_1
J_{13}+ b_2 J_{42}\not =0$ not depending on $Q$.

(iii) If the BC are $y_1(0)=y_1(1)=0$, then
conditions~\eqref{3.7A}--\eqref{3.7B} acquire a simple form
$Q_{12}(0) \cdot Q_{12}(1) \ne 0$ not depending on $Q_{21}$. In
this case the system of root functions of the unperturbed operator
$L_{C,D}(0)$ (with $Q=0$) is incomplete.
       \end{remark}
To prove this theorem, we use the transformation operators
existing for general systems of the form~\eqref{1.1} with $B=B^*$
due to~\cite[Theorem 1]{Mal99}.

   \begin{lemma}\cite{Mal99}
Assume that $e_{\pm}(\cdot;\lambda)$ are solutions of the
system~\eqref{3.1} corresponding to the initial conditions
$e_+(0;\lambda)=\binom{1}{1}, \quad e_-(0;\lambda)=\binom{1}{-1}$.
Then $e_{\pm}(\cdot;\lambda)$ admit the representations
    \begin{equation}\label{3.8}
e_{\pm}(x;\lambda) = (I+K_{\pm})e^0_{\pm}(x;\lambda) =
e^0_{\pm}(x;\lambda) + \int^x_0 K_{\pm}(x,t)
e^0_{\pm}(t;\lambda)dt,
    \end{equation}
where
  \begin{equation*}
e^0_{\pm}(x;\lambda)=\binom{e^{ib_1\lambda x}}{\pm e^{ib_2\lambda
x}}, \qquad K_{\pm}(x,t)=\bigl(K^{\pm}_{ij}(x,t)\bigr)^2_{i,j=1},
    \end{equation*}
and $K^{\pm}_{ij}(\cdot,\cdot)\in W^1_1(\Omega),\  \Omega=\{0\le
t\le x\le 1\}$. Moreover, $K^{\pm}_{ij}\in C^1(\Omega)$ if $Q\in
C(\Omega)\otimes{\Bbb C}^{2\times 2}$.
  \end{lemma}
The following lemma is the key result for proving
Theorem~\ref{th3.1}. It is similar to  the known statement for
Sturm-Liouville operator (cf.~\cite[Lemma 6]{Mal08}).
   \begin{lemma}\label{lem3.2}
Let $Q(\cdot)\in C(\Omega)\otimes{\Bbb C}^{2\times 2}$, and let
$K_{\pm}(\cdot,\cdot)$ be the kernels of the transformation
operators given by~\eqref{3.8}. Then the following relations hold:
    \begin{eqnarray}
K^+_{11}(1,1) - K^-_{11}(1,1) & = & 2i b_1 (b_2 - b_1)^{-1} \cdot b_1 Q_{12}(0),\label{K11-} \\
K^+_{21}(1,1) + K^-_{21}(1,1) & = & 2i b_1 (b_2 - b_1)^{-1} \cdot b_2 Q_{21}(1), \label{K21+} \\
K^+_{21}(1,1) - K^-_{21}(1,1) & = & 0, \label{K21-} \\
K^+_{12}(1,1) - K^-_{12}(1,1) & = & 0, \label{K12-} \\
K^+_{12}(1,1) + K^-_{12}(1,1) & = & 2i b_2 (b_1 - b_2)^{-1} \cdot b_1 Q_{12}(1),\label{K12+} \\
K^+_{22}(1,1) - K^-_{22}(1,1) & = & 2i b_2 (b_1 - b_2)^{-1} \cdot b_2 Q_{21}(0). \label{K22-}
    \end{eqnarray}
   \end{lemma}
\begin{proof}[Proof]
In the case of $Q(\cdot)\in C[0,1]\otimes{\Bbb C}^{2\times 2}$,
the kernels $K^+(\cdot,\cdot)$ of the transformation operators are
related by
   \begin{eqnarray}
B D_x K^{\pm}(x,t)+D_t K^{\pm}(x,t)B=-i Q(x)K^{\pm}(x,t),\quad
(x,t)\in\Omega,
   \end{eqnarray}
and by the boundary conditions
   \begin{gather}
\label{3.12} K^{\pm}_{12}(x,x) = i\frac{b_1 b_2}{b_1 -
b_2}Q_{12}(x), \qquad K^{\pm}_{21}(x,x) = i\frac{b_1
b_2}{b_2-b_1}Q_{21}(x), \\
  \label{3.13} b_2K^{\pm}_{11}(x,0)\pm b_1 K^{\pm}_{12}(x,0)=0, \quad
b_2K^{\pm}_{21}(x,0)\pm b_1 K^{\pm}_{22}(x,0)=0
   \end{gather}
(see~\cite{Mal99}). Relations~\eqref{K21+}--\eqref{K12+} are
immediately implied by~\eqref{3.12}.

Further, the kernels $K^{\pm}(\cdot,\cdot)$ are related by
    \begin{equation}\label{3.14}
K^+(x,t)=K^-(x,t) + \Psi(x-t)+\int^x_t K^-(x,s)\Psi(s-t)ds
    \end{equation}
(see~\cite[formula (1.44)]{Mal99}), where $\Psi(\cdot)$ stands for
the diagonal matrix function,
$\Psi(\cdot)=\diag\bigl(\Psi_1(\cdot), \Psi_2(\cdot)\bigr)\in
C^1[0,1]\otimes{\Bbb C}^{2\times 2}$. It follows
from~\eqref{3.12}--\eqref{3.14} that
    \begin{eqnarray}\label{3.15}
\Psi_1(0) = K^+_{11}(0,0)-K^-_{11}(0,0) = -b_1
b_2^{-1}\bigl(K^+_{12}(0,0) + K^-_{12}(0,0)\bigr)\nonumber  \\
= 2i b^2_1(b_2-b_1)^{-1}Q_{12}(0),
    \end{eqnarray}
  \begin{eqnarray}\label{3.16}
\Psi_2(0) = K^+_{22}(0,0)- K^-_{22}(0,0) = - b^{-1}_1
b_2\bigl(K^+_{21}(0,0) + K^-_{21}(0,0)\bigr) \nonumber \\
= - 2i b^2_2(b_2 - b_1)^{-1}Q_{21}(0).
  \end{eqnarray}
On the other hand, due to~\eqref{3.14} we have
    \begin{equation}\label{3.17}
K^+_{jj}(1,1)-K^-_{jj}(1,1) = \Psi_j(0), \quad  j\in\{1,2\}.
    \end{equation}
Combining ~\eqref{3.15} and~\eqref{3.16} with~\eqref{3.17} we
arrive at relations~\eqref{K11-}, \eqref{K22-}.
   \end{proof}
   \begin{proof}[The proof of Theorem~\ref{th3.1}] (i)
The spectrum $\sigma(L_{C,D})$ of the operator $L_{C,D}$ generated
by problem~\eqref{3.1}--\eqref{3.3} in $L^2([0,1]; {\Bbb C}^2)$
coincides with the zero set of the determinant $\Delta(\cdot)$,
and the multiplicity $p_n$ of the zero $\lambda_n$ of the (entire)
function $\Delta(\cdot)$ coincides with the dimension of the root
subspace
$$
\cH_n := \Span\{\ker(L_{C,D} - \lambda_n)^k: \ k\in{\Bbb Z}_+\},
\qquad \dim \cH_n=p_n
$$
(see~\cite[Sec.5.6]{AgrKatSiVoi99}, \cite{Markus86},
\cite{Nai69}). Let us introduce solutions $w_j(x;\gl)$
of~\eqref{3.1} by setting
   \begin{equation}\label{4.5}
w_1(x;\lambda):= u_{22}(\lambda)\Phi_1 -
u_{21}(\lambda)\Phi_2,\quad w_2(x;\lambda):=
-u_{12}(\lambda)\Phi_1 + u_{11}(\lambda)\Phi_2,
   \end{equation}
where $u_{j1}(\cdot), u_{j2}(\cdot)$ are entries of the matrix
$U(\cdot)$ of the form~\eqref{3.3A}. Clearly, $U_j(w_j) =
\Delta(\lambda)$ and $U_1(w_2) = U_2(w_1) =0$; in particular,
$U_j\bigl(w_j(\cdot;\lambda_n)\bigr)= \Delta(\lambda_n)=0$.
Further, the functions $w^{(k)}_j(x;\lambda):=
D^k_{\lambda}w_j(x;\lambda)$ satisfy the equations
   \begin{equation}\label{4.6}
Lw^{(k)}_j = \lambda w^{(k)}_j + kw_j^{(k-1)}, \qquad j\in\{1,2\}.
   \end{equation}
Since $U_i(D^k_{\lambda}w_j(x;\lambda)) =
D^k_{\lambda}(U_i(w_j(x;\lambda)))$ and  $\lambda_n$ is the root
of characteristic determinant $\Delta(\cdot)$ of multiplicity
$p_n$, then the functions
$D^k_{\lambda}w(x;\lambda)|_{\lambda=\lambda_n}$,\
$k\in\{1,\ldots,p_n\}$, satisfy boundary conditions~\eqref{3.3} as
well. Hence in the case of $\dim\ker(L_{C,D}-\lambda_n)=1$, at
least one of the two systems
$\{w_j^{(k)}(\cdot;\lambda)\}_{k=1}^{p_n}$,\ $j\in \{1,2\}$, forms
a chain of an eigenfunction and associated functions.

If $\dim\ker(L_{C,D}-\lambda_n)=2$, the root subspace $\cH_n$ has
the form
    \begin{equation}\label{4.7A}
\cH_n  = \Span\{D^k_{\lambda}w_j(x;\gl)|_{\gl = \gl_n}, \
k\in\{0,1,\ldots,p_{n}-1\}, \ j\in\{1,2\}\}.
    \end{equation}

By assuming that the system of root functions of the operator
$L_{C,D}$ is incomplete in $L^2([0,1]; {\Bbb C}^2)$, we find a
vector $(0\not =)f = \col(f_1,f_2)$ orthogonal to this system.
Hence we conclude that the entire functions
   \begin{equation}\label{5.24}
w_j(\lambda;f) := \int^1_0\langle w_j(x;\lambda),f(x)\rangle dx
\quad j\in\{1,2\},
   \end{equation}
have a zero of multiplicity $\ge p_n$ at every point
$\lambda_n\in\sigma(L_{C,D})$. Thus,  $G_j(\cdot; f) := w_j(\cdot;
f)/\Delta(\cdot), \ j\in \{1,2\},$ is the entire function. Let us
estimate their growth.

(ii)\ First we estimate the growth of $\Delta(\cdot)$ from below. Since
$\Phi(0;\lambda)=I_2$ and $e_{\pm}(0;\lambda) = \binom{1}{\pm 1}$ due
to~\eqref{3.8}, we have
$$
2\Phi_1(\cdot;\lambda) = e_+(\cdot;\lambda) + e_-(\cdot;\lambda),
\quad 2\Phi_2(\cdot;\lambda) = e_+(\cdot;\lambda) -
e_-(\cdot;\lambda).
$$
By setting
    \begin{equation}
R_{jk}^{\pm}(t):=  K^+_{jk}(1,t) \pm K^-_{jk}(1, t), \quad
j,k\in\{1,2\},
    \end{equation}
and by taking into account representations~\eqref{3.8} for the
solutions $e_{\pm}(\cdot;\lambda)$, we obtain
\begin{eqnarray}
2\varphi_{11}(1;\lambda) & = & 2 e^{i b_1 \lambda} + \int^1_0 R_{11}^{+}(t)e^{ib_1\lambda t}dt
+ \int^1_0 R_{12}^{-}(t)e^{ib_2\lambda t}dt, \label{3.18a} \\
2\varphi_{12}(1;\lambda) & = & \ \ \ \ \ \ \ \ \ \ \ \int^1_0 R_{11}^{-}(t)e^{ib_1\lambda t}dt
+ \int^1_0 R_{12}^{+}(t)e^{ib_2\lambda t}dt. \label{3.19} \\
2\varphi_{21}(1;\lambda) & = & \ \ \ \ \ \ \ \ \ \ \ \int^1_0 R_{21}^{+}(t)e^{ib_1\lambda t}dt
+ \int^1_0 R_{22}^{-}(t)e^{ib_2\lambda t}dt, \label{3.18} \\
2\varphi_{22}(1;\lambda) & = & 2 e^{i b_2 \lambda} + \int^1_0 R_{21}^{-}(t)e^{ib_1\lambda t}dt
+ \int^1_0 R_{22}^{+}(t)e^{ib_2\lambda t}dt. \label{3.19a}
\end{eqnarray}
Noting that $R_{jk}^{\pm}(\cdot)\in C^1[0,1],$\  $j,k\in \{1,2\}$,
we integrate by parts in~\eqref{3.18a},~\eqref{3.18},~\eqref{3.19}
and~\eqref{3.19a} and insert the expressions thus obtained
into~\eqref{3.5}. Then we arrive at the following expression for
the characteristic determinant
    \begin{eqnarray}
\Delta(\lambda) = J_{12} + J_{34}e^{i(b_1+b_2)\lambda} &+&
\left(J_{32}+\frac{r_1(1)}{2ib_1\lambda}\right) e^{ib_1\lambda} +
\left(J_{14}+\frac{r_2(1)}{2ib_2\lambda}\right) e^{ib_2\lambda} \nonumber  \\
- \frac{r_1(0)}{2ib_1 \lambda}  - \frac{r_2(0)}{2i b_2\lambda}
&-&\int^1_0 r'_1(t)\frac{e^{ib_1\lambda t}}{2ib_1\lambda}dt -
\int^1_0 r'_2(t)\frac{e^{ib_2\lambda t}}{2ib_2\lambda}dt
    \end{eqnarray}
in which
   \begin{eqnarray*}
r_1(t) := J_{32}R_{11}^{+}(t) + J_{13}R_{11}^{-}(t) + J_{42}
R_{21}^{+}(t) + J_{14}R_{21}^{-}(t), \\
r_2(t) := J_{32}R_{12}^{-}(t) + J_{13}R_{12}^{+}(t) + J_{42}
R_{22}^{-}(t) + J_{14}R_{22}^{+}(t).
   \end{eqnarray*}
By Lemma~\ref{lem3.2},
   \begin{eqnarray}
J_{32} + \frac{r_1(1)}{2 i b_1 \lambda} = J_{32} \left(1 + \frac{R_{11}^{+}(1)}{2 i b_1 \lambda}\right)
+ \frac{b_1 J_{13} Q_{12}(0) + b_2 J_{42} Q_{21}(1)}{(b_2 - b_1) \lambda}, \\
J_{14} + \frac{r_2(1)}{2 i b_2 \lambda} = J_{14} \left(1 + \frac{R_{22}^{+}(1)}{2 i b_2 \lambda}\right)
+ \frac{b_1 J_{13} Q_{12}(1) + b_2 J_{42} Q_{21}(0)}{(b_1 - b_2) \lambda}.
   \end{eqnarray}
Conditions~\eqref{3.7A}--\eqref{3.7B} yield now that
$$
\left|J_{32}+\frac{r_1(1)}{2ib_1\lambda}\right| \ge \frac{c}{|\lambda|+1},\qquad
\left|J_{14}+\frac{r_2(1)}{2ib_2\lambda}\right| \ge \frac{c}{|\lambda|+1},
\qquad c>0, \ \ \lambda \in \mathbb{C}\backslash\{0\}.
$$
This implies the desired estimates for $\Delta(\cdot)$ from below:
   \begin{equation}
|\Delta(\lambda)|\ge \frac{c}{|\lambda|+1}\exp\left(|b_{\mp}\Im\lambda|\right),
\quad \lambda\in\Omega^{\pm}_{\varepsilon} := \{\lambda:
\varepsilon\le\pm\arg\lambda\le\pi-\varepsilon\},
   \end{equation}
where $b_- := b_1,\  b_+ :=  b_2$.

(iii) In this step we estimate the growth of $w_j(\cdot;f)$ from
above. We show that
    \begin{equation} \label{5.31A}
w_j(x;\lambda) = O(\exp\left(|b_{\mp}\Im\lambda|\right)),\qquad
\lambda\in\Omega^{\pm}_{\varepsilon}.
    \end{equation}
Let $Y_j:=\col(y_{1j},y_{2j}),\  j\in\{1,2\}$, be the solution of
\eqref{3.1} satisfying \eqref{2.21}, i.e.
     \begin{equation}\label{5.31}
y_{kj}(x,\lambda)=\bigl(\delta^j_k + o(1)\bigr)\exp\left(ib_j\lambda x\right),
\qquad \lambda\in\Omega^+_{\varepsilon},\quad j,k\in\{1,2\},
     \end{equation}
and let
$\wt{U}(\lambda):=\bigl(\wt{u}_{jk}(\lambda)\bigr)^2_{j,k=1} :=
\bigl(U_j(Y_k)\bigr)^2_{j,k=1}$. Alongside solutions \eqref{4.5}
we introduce solutions
    \begin{equation}\label{5.33A}
V_1(x,\lambda) = {\wt u}_{22}(\lambda)Y_1  - {\wt
u}_{21}(\lambda)Y_2,\quad V_2(x,\lambda) = -{\wt
u}_{12}(\lambda)Y_1 + {\wt u}_{11}(\lambda)Y_2.
    \end{equation}
According to \eqref{2.22} and \eqref{2.25} the fundamental
matrices $\Phi(x, \cdot)$ and $Y(x, \cdot) =
\bigl(Y_1(\cdot,\lambda)\  Y_2(x, \cdot)\bigr)$ of equation
\eqref{3.1}  as well as  the matrices $U(\cdot)$ and
$\wt{U}(\cdot)$ are connected by
   \begin{equation}\label{5.33}
\Phi(x,\lambda)= Y(x,\lambda)P(\lambda)  \quad
           \text{and}\quad  U(\lambda)=\wt{U}(\lambda)P(\lambda),\qquad
           \lambda\in\Omega^+_{\varepsilon},
   \end{equation}
where $P(\cdot)$  is the invertible holomorphic $2\times 2$ matrix
function.  Hence (cf. \eqref{2.27})
   \begin{equation}\label{5.34}
w_j(x,\lambda)=V_j(x,\lambda)\det P(\lambda), \qquad
\lambda\in\Omega^+_{\varepsilon},\quad j\in\{1,2\}.
   \end{equation}
It follows from \eqref{5.31} that
    \begin{equation*}
\wt{u}_{11}(\lambda)= O(e^{ib_1\lambda}),\quad
\wt{u}_{12}(\lambda)= O(1), \quad  \wt{u}_{21}(\lambda)=
O(e^{ib_1\lambda}),\quad \wt{u}_{22}(\lambda)= O(1),
    \end{equation*}
as $\lambda\to \infty,\ \lambda\in\Omega^+_{\varepsilon}.$ It
follows with account of \eqref{5.31} and  \eqref{5.33A} that
    \begin{equation}\label{5.36}
V_1(x,\lambda)= O(e^{ib_1\lambda}), \qquad V_2(x,\lambda)=
O(e^{ib_1\lambda})\quad\text{as}\quad     \lambda\to\infty, \quad
\lambda\in\Omega^+_{\varepsilon}.
    \end{equation}
Moreover, substituting $x=0$  in the first of equalities \eqref{5.33}  and
taking into account   \eqref{5.31} and $\Phi(0,\lambda)= I_2$, we get
$P(\lambda) = I_2 + o_2(\lambda)$. Combining this relation with \eqref{5.34}
and \eqref{5.36}  yields \eqref{5.31A} for $\lambda\in\Omega^+_{\varepsilon}.$
The second relation in \eqref{5.31A} is proved similarly. In turn, combining
estimates \eqref{5.31A} with \eqref{5.24} yields
    \begin{equation} \label{5.31B}
w_j(\lambda; f) = o(\exp\left(|b_{\mp}\Im\lambda|\right))\qquad \text{as} \quad
\lambda \to\infty, \quad \lambda\in\Omega^{\pm}_{\varepsilon}.
    \end{equation}

Hence applying the Phragmen-Lindel\"of theorem  to the functions $G_j(\cdot;f)$
in the angles $\Omega^{\pm}_{\varepsilon}$, we conclude that $G_j(\cdot;f) =
\const,\ j\in\{1,2\}$. Using the same technique as in~\cite{Mal08} one can
prove that $G_j(\cdot;f) = 0,\ j\in\{1,2\}$. Now the proof is completed by
applying steps (iv) and (v) of the proof of Theorem~\ref{th2.1}. The minimality
is implied by Lemma \ref{Lem_minimality}.
        \end{proof}

       \begin{remark}
In the case of Dirac system ($-b_1= b_2\in \Bbb R_+$) the step (iii) of the
proof can be substantially simplified.  To this end we set
   \begin{equation}
\Phi_{j r}(x,\lambda) := \varphi_{2j}(1,\lambda)\Phi_1(x,\lambda) -
\varphi_{1j}(1,\lambda)\Phi_2(x,\lambda), \qquad j\in\{1,2\}.
   \end{equation}
Clearly, $\Phi_{j r}(\cdot,\lambda)$ is the solution of equation
\eqref{3.1}.  Moreover,  since $\tr Q(x)=0,x\in[0,1]$, by
Liouville theorem  $\det\Phi(x,\lambda) = \det\Phi(0,\lambda) =
1,\ x\in[0,1]$. Hence
    \begin{equation}
\Phi_{1r}(1,\lambda)=\binom{0}{-1}\qquad \text{and}\qquad
\Phi_{2r}(1,\lambda)=\binom{1}{0}.
    \end{equation}
On the other hand,
    \begin{eqnarray*}
w_j(x,\lambda) = [a_{j1}\varphi_{12}(0,\lambda) + a_{j2}\varphi_{22}(0,\lambda)
+ a_{j3}\varphi_{12}(1,\lambda)
+ a_{j4}\varphi_{22}(1,\lambda)]\Phi_1(x,\lambda)   \nonumber  \\
- [a_{j1}\varphi_{11}(0,\lambda) + a_{j2}\varphi_{21}(0,\lambda) +
a_{j3}\varphi_{11}(1,\lambda)
+ a_{j4}\varphi_{21}(1,\lambda)]\Phi_2(x,\lambda) \nonumber  \\
=a_{j2}\Phi_1(x,\lambda)-a_{j1}\Phi_2(x,\lambda)+a_{j3}\Phi_{1r}(x,\lambda)+a_{j4}\Phi_{2r}(x,\lambda).
    \end{eqnarray*}
Combining this representation with \eqref{5.24} we arrive at \eqref{5.31A}.
      \end{remark}

      \begin{corollary}\label{cor_irreg_adj}
Assume  the conditions of Theorem \ref{th3.1}. Then  the system of root
functions of the  operator $L_{C,D}^*$ is also complete and minimal in
$L^2\left([0,1];{\Bbb C}^2\right)$.
   \end{corollary}
   \begin{proof}[Proof]
If $J_{32} \ne 0$ and $J_{14} \ne 0$ then Corollary~\ref{cor2.1A}
is applicable. Now let $J_{32} \cdot J_{14} = 0$. Then one of the
conditions~\eqref{3.7AA} or~\eqref{3.7BB} holds.

Hence either $J_{13} \ne 0$ or $J_{42} \ne 0$. Without loss of
generality we assume that $J_{13} \ne 0$. Let $a_j:=\col\
(a_{1j},a_{2j}),\  j\in\{1,\ldots,4\}$. Then $J_{13} \ne 0$
implies  $a_1 \ne 0$ and $a_3 \ne 0$. Now we consider three cases.

\emph{(i)} $J_{14}=J_{32}=0$. Then conditions~\eqref{3.7AA} and~\eqref{3.7BB}
hold true. Since $a_1 \ne 0$ and $a_3 \ne 0$ then $a_4=\alpha_1 a_1$ and
$a_2=\alpha_2 a_3$ with some $\alpha_1, \alpha_2\in \mathbb{C}$. Hence
conditions~\eqref{3.3} are equivalent to the following ones:
     \begin{equation}\label{4.2}
y_1(0) = -\alpha_1 y_2(1), \quad   y_1(1) = -\alpha_2 y_2(0),
     \end{equation}
It can easily be seen that the adjoint operator $L^*_{C,D} :=
(L_{C,D})^*$ is defined by the differential expression $L^*=-i B
d/dx+Q^*(x)$, where
$$
Q^*(x)= \left(
\begin{array}{cc}
  0 & Q^*_{12}(x) \\
  Q^*_{21}(x) & 0 \\
\end{array}%
\right)=:\left(
\begin{array}{cc}
  0 & Q_{12*}(x) \\
  Q_{21*}(x) & 0 \\
\end{array}%
\right),$$
and the boundary conditions
       \begin{equation}\label{4.20}
U_{1*}(y) := {\overline\alpha}_1 b_2 y_1(0) + b_1y_2(1) = 0,\quad U_{2*}(y) :=
b_1 y_2(0) + {\overline\alpha}_2 b_2 y_1(1) = 0.
       \end{equation}
It follows from~\eqref{4.2} and~\eqref{4.20} that
    \begin{equation}\label{4.21}
J_{42*} = b^2_1=b^2_1{\overline J}_{13}\quad \text{and} \quad  J_{13
*}=b^2_2{\overline \alpha}_1{\overline\alpha}_2=b^2_2{\overline J}_{42}.
    \end{equation}
Now we check conditions~\eqref{3.7AA}, \eqref{3.7BB} for the operator
$L^*_{C,D}$. Due to~\eqref{4.21}, expressions~\eqref{3.7AA}, \eqref{3.7BB} for
$L^*_{C,D}$ are of the form
    \begin{eqnarray*}
\ b_1 J_{13 *}Q_{12*}(0) + b_2 J_{42*}Q_{21 *}(1) = b_1 b_2\bigl[
b_2\overline {J}_{42}\overline {Q_{21}(0)} + b_1{\overline J_{13}}\overline {Q_{12}(1)}
\bigr],  \\
\ b_2 J_{42*}Q_{21 *}(1)+b_1 J_{13 *}Q_{12 *}(0)=b_1 b_2\bigl[
b_2\overline {J}_{42}\overline {Q_{21}(1)} + b_1{\overline J_{13}}\overline {Q_{12}(0)}
\bigr],
      \end{eqnarray*}
and different form zero by the assumptions  of  Theorem \ref{th3.1}.

\emph{(ii)} $J_{32}=0, J_{14} \ne 0$.  Then condition~\eqref{3.7AA} hold true.
Since $a_3 \ne 0$ the condition $J_{32}=0$ means that $a_2=\alpha a_3$ with
some $\alpha \in{\Bbb C}$. Since $J_{14} \ne 0$ we represent boundary
conditions~\eqref{3.3} as
   \begin{equation*}
\binom{y_1(0)}{y_2(1)}=-
\begin{pmatrix}
a_{11}&a_{14}\\
a_{21}&a_{24}
\end{pmatrix}^{-1}
    \begin{pmatrix}
\alpha a_{13}& a_{13}\\
\alpha a_{23}& a_{23}
    \end{pmatrix}
\binom{y_2(0)}{y_1(1)}=
\begin{pmatrix}
\alpha \beta_1 & \beta_1\\
\alpha \beta_2 & \beta_2
\end{pmatrix}
\binom{y_2(0)}{y_1(1)},
   \end{equation*}
where $\beta_1 := -J^{-1}_{14}J_{24}$ and $\beta_2 :=
-J^{-1}_{14}J_{12}$. Thus, conditions~\eqref{3.3} take the form
    \begin{eqnarray}\label{4.20A}
U_1(y) := 1\cdot y_1(0) - \alpha \beta_1 \cdot y_2(0) - \beta_1 \cdot y_1(1) +
0\cdot
y_2(1)=0, \nonumber \\
U_2(y):=0\cdot y_1(0)-\alpha \beta_2 \cdot y_2(0)-\beta_2 \cdot y_1(1)+1\cdot
y_2(1)=0.
    \end{eqnarray}
Now boundary conditions for the adjoint operator $L^*_{C,D}$ are rewritten as
follows:
\begin{eqnarray}\label{4.21A}
U_{1*}(y) = -b^{-1}_1\overline{\beta_1}y_1(0) + 0\cdot y_2(0) + b^{-1}_1
y_1(1) + b^{-1}_2\overline{\beta}_2 y_2(1) = 0,\nonumber \\
U_{2*}(y):= b^{-1}_1\overline{\beta}_1 \overline{\alpha}y_1(0) + b^{-1}_2
y_2(0) + 0\cdot y_2(1) - b^{-1}_2\overline{\beta_2}\overline{\alpha}y_2(1)=0.
 \end{eqnarray}
Both relations~\eqref{4.20A} and~\eqref{4.21A} yield that $J_{14*}=0, J_{32*} =
-b^{-1}_1 b^{-1}_2\not =0$ and
    \begin{equation}\label{4.22A}
b_1 J_{13*}=-b^{-1}_1\overline{\beta}_1\overline{\alpha}=b^{-1}_1\overline{J_{42}}, \quad
b_2 J_{42*}= -b^{-1}_2\overline{\beta}_2= b^{-1}_2\overline{J}_{13}.
    \end{equation}
The equations thus obtained allow us to prove that the condition~\eqref{3.7BB}
for $L^*_{C,D}$ is equivalent to the conditions~\eqref{3.7AA} for $L_{C,D}$.
Indeed, taking account of relations $Q_{ij*}(x)= \overline{Q_{ji}(x)}, \ i\not
=j$, and~\eqref{4.22A}, we get
    \begin{eqnarray}
b_1 J_{13*}Q_{12*}(1) + b_2 J_{42*} Q_{21*}(0) =
b^{-1}_1\overline{J}_{42}\overline{Q_{21}(1)} + b^{-1}_2 \overline{J}_{13}\overline{Q_{12}(0)} \nonumber  \\
=b^{-1}_1 b^{-1}_2 \bigl[
b_1\overline{J}_{13}\overline{Q_{12}(0)} + b_2\overline{J}_{42}\overline{Q_{21}(1)}
\bigr]\not =0.
  \end{eqnarray}

\emph{(iii)} $J_{32} \ne 0, J_{14} = 0$. This case is similar to (ii).

Thus, in all cases the assumptions of Theorem~\ref{th3.1} hold true for the
adjont operator $L^*_{C,D}$, and hence the system of its root functions is
complete and minimal in $L^2\bigl([0,1];{\Bbb C}^2\bigr)$.
     \end{proof}
   \begin{corollary}
Suppose that the operator $L_{C,D}(0)$ of the
form~\eqref{3.1}--\eqref{3.3} is dissipative. If $Q\in
C[0,1]\otimes{\Bbb C}^{2\times 2}$ and condition \eqref{3.7AA} is
fulfilled, then the systems of root functions of both operators
$L_{C,D}(Q)$ and $L^*_{C,D}(Q)$ are complete and minimal in
$L^2\bigl([0,1];{\Bbb C}^2\bigr)$.
     \end{corollary}
  \begin{proof}[Proof]
Since $L_{C,D}(0)$ is dissipative, it follows from
Lemma~\ref{lem_dissip}  that the condition $J_{14}=\det T_-\not
=0$ is met. It suffices to apply Theorem~\ref{th3.1} and
Corollary~\ref{cor_irreg_adj}.
   \end{proof}

      \begin{remark}
\emph{Dissipative boundary conditions for the equation~\eqref{3.1}
are always nondegenerated}. But, as distinguished from the case of
the Sturm-Liouville operator, they are \emph{not necessarily
regular} because they do no guarantee the validity of the first
regularity condition in~\eqref{2.4}. Moreover, even in the case of
$Q=Q^*$, the condition $\det T_1\not=0$ is not necessary for the
completeness of the system of root functions of the dissipative
operator $L_{C,D}(Q)$.

Note else that there exist non-Volterra dissipative operators
$L_{C,D}(Q)$ for which the system of root functions is not
necessarily complete in $L^2\bigl([0,1];{\Bbb C}^2\bigr)$.
      \end{remark}

Next we consider boundary conditions~\eqref{3.3} of the special form
 \begin{equation}\label{5.3}
U_1(y):=y_1(0)-\beta_1 y_2(0)=0,\quad  U_2(y):=y_2(1) -\beta_2
y_2(0) = 0.
 \end{equation}
   \begin{corollary}\label{cor3.3}
Suppose that $Q\in C[0,1]\otimes{\Bbb C}^{2\times 2}$,\
$\beta_1\in \Bbb C\setminus \{0\}$ and $L_{C,D}$ is the operator
of form~\eqref{3.1}--\eqref{3.3}, where $U_1$ and  $U_2$ are
defined by \eqref{5.3}. Then:
\item $(i)$ the operator $L_{C,D}$ is dissipative whenever  $\im
Q(x)\ge 0$ and
   \begin{equation}\label{5.4}
b^{-1}_2 |\beta_2|^2 \le  b^{-1}_2 + b^{-1}_1 |\beta_1|^2;
   \end{equation}
\item (ii) if $Q_{21}(1)\not =0$, then the system of root
functions of the operator $L_{C,D}$ is complete and minimal;
   \end{corollary}

   \begin{proof}[Proof]

   (ii) If $Q_{21}(1)\not =0$, then Theorem~\ref{th3.1} is applicable, since in this
case we have $J_{32}=0=J_{13}$ but $J_{14}=1$ and $J_{42}=\beta_1\not =0$.
   \end{proof}

   \begin{remark}
(i) We emphasize that for $Q_{21}(1)\not =0$ the completeness (and
the minimality) of the EAF system of the operator $L_{C,D}(Q)$
holds in the assumptions of Corollary~\ref{cor3.3} with
$\beta_2=0$ too. In the latter case  the second of the
conditions~\eqref{5.3} is "of Volterra type"\ and the
corresponding operator  $L_{C,D}(0)$ with $Q=0$ is incomplete.
Moreover, for $Q=0$ the operator $L_{C,D}(0)$ has a Volterra
inverse.
    \end{remark}
     \begin{remark}
(i) Theorem  \ref{th3.1}  might be considered as an analog of a
special case of the  completeness  result on BVP for
Sturm-Liouville operators with degenerate BC (see \cite[Theorem
1]{Mal08}). More general result even for $n\times n$ Dirac type
systems that involves considerations of derivatives of a smooth
potential matrix $Q$ is more complicated and will be considered in
the forthcoming  paper \cite{LunMal2011}.

(ii)  In connection with Theorem  \ref{th3.1}  and other results
of this section  we mention the papers
~\cite{TroYam01},~\cite{TroYam02},~\cite{HasOri09} devoted to the
Riesz basis property of EAF for BVP with separated (and hence
strictly regular) BC for $2\times 2$ Dirac systems
(\cite{TroYam01},~\cite{TroYam02},~\cite{Mit04}) and  for $2\times
2$ Dirac type systems (\cite{HasOri09}).

The Riesz basis property of EAF for BVP with \emph{regular but
non-strictly regular} (including periodic, antiperiodic and other)
BC for $2\times 2$ Dirac systems have been investigated by
P.~Djakov and
B.~Mityagin~\cite{Mit04},~\cite{DjaMit09},~\cite{DjaMit10}.
Namely, in~\cite{Mit04} and ~\cite{DjaMit09} they  proved the
Riesz basis property of subspaces (spectral projections) for  $2
\times 2$ Dirac system with periodic and antiperiodic BC. In the
next publication~\cite{DjaMit10} these authors extended their
result to the case of \emph{arbitrary regular} but not strictly
regular BC. Moreover, in~\cite{DjaMit10} they proved the Riesz
basis property of the system of EAF for BVP with \emph{general
strictly regular} BC under the assumption $Q_{12},Q_{21} \in
L^2[0,1]$.
\end{remark}

\subsection{Necessary conditions of completeness}

Here we complete  Theorem~\ref{th3.1} by the following result on
necessary conditions of completeness which demonstrate that
conditions~\eqref{3.7A}, \eqref{3.7B} for the Dirac system are
sharp.
   \begin{proposition}\label{prop5.12}
Assume that $B=\diag(-1,1),\  J_{14}=J_{32}=0$ but $J_{13}
J_{42}\not =0$. Further, let $0\notin\supp P_1\cup\supp P_2$,
where
    \begin{equation}\label{4.1}
P_1(x):=J_{13}Q_{12}(x)-J_{42}Q_{21}(1-x), \quad P_2(x) :=
J_{13}Q_{12}(1-x) - J_{42}Q_{21}(x).
    \end{equation}
Then the defect of the system of root functions of
problem~\eqref{3.1}--\eqref{3.3} in $L^2[0,1]\otimes{\Bbb C}^2$ is
infinite.
    \end{proposition}

   \begin{proof}[Proof]
By assumption, there exists an $\varepsilon>0$ such that
   \begin{equation}\label{4.1A}
 P_1(x) = P_2(x) = 0,\qquad   x\in[0,\varepsilon]\cup[1-\varepsilon,1].
   \end{equation}
Let $w_j:=\col(w_{j1}, w_{j2})$ be defined by~\eqref{4.5}. Since
$J_{14}=J_{32}=0$ and $J_{13} J_{42} \ne 0$, we conclude that the boundary
conditions~\eqref{3.3} are equivalent to the following ones
\begin{equation}\label{cond_J14=J32=0}
y_1(0)=-\alpha_1 y_2(1), \quad y_2(0)=-\alpha_2 y_1(1),
\end{equation}
where $\alpha_1 \ne 0$ and $\alpha_2 \ne 0$.
Denote
    \begin{equation}\label{4.3}
z_j(x;\lambda ) := \binom{z_{j1}(x;\lambda)}{z_{j2}(x;\lambda)} :=
\binom{-\alpha_1 w_{j2}(1-x;\lambda)}{-\alpha_2
w_{j1}(1-x;\lambda)}, \qquad j\in\{1,2\}.
    \end{equation}
Let us demonstrate that, for $x\in[0,\varepsilon]$ and every $k\in
\Bbb N$, the functions $z^{(k)}_j(x;\lambda) :=
D^k_{\lambda}z_j(x;\lambda), j\in\{1,2\}$, alongside with the
functions $w^{(k)}_j(x;\lambda)$, satisfy the
equation~\eqref{4.6}. Indeed, from~\eqref{4.6} and with account
of~\eqref{4.1A} and~\eqref{4.3} we obtain:
     \begin{eqnarray*}\label{4.14}
L z_j^{(k)} = -i B \frac{d}{dx}z^{(k)}_j + Q(x)z_j^{(k)} =
i\frac{d}{dx}\binom{-z^{(k)}_{j1}(x;\lambda)}{z^{(k)}_{j2}(x;\lambda)}
+\binom{Q_{12}(x)z^{(k)}_{j2}(x;\lambda)}{Q_{21}(x)z^{(k)}_{j1}(x;\lambda)} \\
= i \frac{d}{dx}\binom{-\alpha_1 w^{(k)}_{j2}(1-x;
\lambda)}{\alpha_2 w^{(k)}_{j1}(1-x;\lambda)} +
\binom{\alpha_1\alpha^{-1}_2 Q_{21}(1-x)(-\alpha_2)
w^{(k)}_{j1}(1-x;\lambda)}{\alpha_2\alpha^{-1}_1
Q_{12}(1-x)(-\alpha_1) w^{(k)}_{j2}(1-x;\lambda)}\nonumber \\
= -\lambda\binom{\alpha_1 w^{(k)}_{j2}(1-x;\lambda)}{\alpha_2
w^{(k)}_{j1}(1-x;\lambda)} - k D^{k-1}_{\lambda} \binom{\alpha_1
w_{j2}(1-x;\lambda)}{\alpha_2 w_{j1}(1-x;\lambda)}  \\ = \lambda
z^{(k)}_j(x;\lambda) + kz^{(k-1)}_j(x;\lambda), \qquad
x\in[0,\varepsilon], \quad j\in\{1,2\}.
       \end{eqnarray*}
Further, since $w^{(k)}_j(x;\lambda_n) =
D^k_{\lambda}w(x;\lambda)|_{\lambda=\lambda_n}$, \ $k\in\{1,\ldots,p_n\}$
satisfy the boundary conditions~\eqref{cond_J14=J32=0}, then
from~\eqref{cond_J14=J32=0} and~\eqref{4.3} we obtain that
    \begin{equation*}\label{4.4}
z^{(k)}_{j1}(0;\lambda_n) = -\alpha_1 w^{(k)}_{j2}(1;\lambda_n) =
w^{(k)}_{j1}(0;\lambda_n), \quad z^{(k)}_{j2}(0;\lambda_n) =
w^{(k)}_{j2}(0;\lambda_n),
    \end{equation*}
for $j\in\{1,2\}$ and $n\in \Bbb N$. Therefore, by the uniqueness
theorem we have
    \begin{equation}\label{4.7}
D^k_{\lambda}z_j(x;\lambda)|_{\lambda=\lambda_n}  =
D^k_{\lambda}w_j(x;\lambda)|_{\lambda=\lambda_n},
           \quad x\in[0,\varepsilon],\ k\in\{0,1,\ldots, p_n-1\}.
    \end{equation}
Further, let $f=\col(f_1,f_2)\in L^2[0,1]\otimes{\Bbb C}^2,\
 f(x)=0$ for $x\in[\varepsilon,1-\varepsilon]$ and
    \begin{equation}\label{4.8}
f_1(x) = \overline{\alpha^{-1}_1} f_2(1-x), \quad  f_2(x) =
\overline{\alpha^{-1}_2}f_1(1-x),\quad  x\in[0,\varepsilon].
    \end{equation}
Let us show that $f$ is orthogonal to the system of root functions
of the problem. Taking account of~\eqref{4.7A}, \eqref{4.3},
\eqref{4.7} and \eqref{4.8}, we obtain
   \begin{eqnarray*}
\int^1_0\bigl\langle w^{(k)}_j(x;\lambda_n),f(x)\bigr\rangle dx =
\int^{\varepsilon}_0\bigl\langle w^{(k)}_j(x;\lambda_n),f(x)\bigr\rangle dx\\
+ \int^{\varepsilon}_0\bigl\langle w^{(k)}_j(1-x;\lambda_n),
f(1-x)\bigr\rangle dx = \int^{\varepsilon}_0
w^{(k)}_{j1}(x;\lambda_n)[\overline{f_1(x)} - \alpha^{-1}_1\overline{f_2(1-x)}]dx          \\
+ \int^{\varepsilon}_0
w^{(k)}_{j2}(x;\lambda_n)[\overline{f_2(x)}-\alpha^{-1}_2\overline{f_1(1-x)}]dx
= 0,\quad  n\in{\Bbb N}.
   \end{eqnarray*}
It follows that the defect of the system of root functions is
infinite.
   \end{proof}
   \begin{remark}
Proposition \ref{prop5.12}  is similar to that of
\cite[Proposition 9]{Mal08} for the Sturm-Liouville operator with
degenerate boundary conditions.
      \end{remark}

\section{Completeness of irregular BVP for $2\times 2$ systems with $B\not = B^*$}

Consider system~\eqref{3.1} with the matrix $B=\diag (b_1^{-1},
b_2^{-1})\not = B^*$ assuming that $b_1/b_2\notin{\Bbb R}$. In
this case the lines $\{\lambda \in \mathbb{C} : \Re(i b_j
\lambda)=0\}$, $j \in \{1,2\},$ divide
 the complex plane  in  two pairs of vertical
sectors and Corollary \ref {cor2.1} guarantees the completeness
and the minimality of the root system of problem~\eqref{3.1},
\eqref{3.3} in the following cases:
    \begin{equation}\label{5.1}
(i) \quad  J_{14}J_{23} \not = 0\quad \text{and}\quad (ii)\quad
J_{12}J_{34}\not =0.
    \end{equation}
Here we consider equation~\eqref{3.1} subject to the boundary
conditions
\begin{equation}\label{5.2}
U_1(y) := y_1(0) - h_0 y_2(0)=0 \qquad U_2(y) := y_1(1) - h_1
y_2(0)=0,
\end{equation}
where $h_0 h_1 \ne 0$.
In this case, $J_{14} = J_{34}=0$ and conditions~\eqref{5.1} are
violated. However, the following result holds.
   \begin{theorem}\label{th7.1}
Let $B = \diag(b^{-1}_1,b^{-1}_2)$ and $a:= b_1
b^{-1}_2\notin{\Bbb R}$. Then the system of root functions of
problem~\eqref{3.1}, \eqref{5.2} is complete and minimal in
$L^2\bigl([0,1];{\Bbb C}^2\bigr)$.
   \end{theorem}

\begin{proof}[Proof]
The line $\{\lambda \in \mathbb{C} : \Re ({ib_1}\lambda) = \Re
({ib_2}\lambda)\}$ divides the complex plane into two half-planes.
By Proposition~\ref{BirkSys}, in each of these half-planes
equation~\eqref{3.1} has the fundamental system of solutions
$\{Y_1(x;\lambda),Y_2(x;\lambda)\}$ satisfying the asymptotics
  \begin{equation}\label{syst.6}
  Y_1(x;\lambda)=\binom{e^{i{b_1}\lambda x}(1+o(1))}{e^{i{b_1}\lambda x}o(1)}
   \qquad \text{and} \qquad
  Y_2(x;\lambda)=\binom{e^{i{b_2}\lambda x}o(1)}{e^{i{b_2}\lambda x}(1+o(1))},
  \end{equation}
as $\lambda\to\infty$  uniformly with respect to $x\in [0,1]$. In
particular, in  these half-planes,
  \begin{equation}\label{syst.7}
  Y_1(0;\lambda)=\binom{1+o(1)}{o(1)}
   \quad \text{and} \quad Y_2(0;\lambda)=\binom{o(1)}{1+o(1)}  \quad \text{as}\quad  \lambda\to\infty.
  \end{equation}
Let  $\displaystyle
\Phi_1(x;\lambda)=\binom{\varphi_{11}(x;\lambda)}{\varphi_{21}(x;\lambda)}$
and
$\displaystyle\Phi_2(x;\lambda)=\binom{\varphi_{12}(x;\lambda)}{\varphi_{22}(x;\lambda)}$
stand for the solutions of the Cauchy problem for
system~\eqref{3.1} satisfying the initial conditions
  \begin{equation}\label{9}
  \Phi_1(0;\lambda)=\binom{1}{0}
  \qquad \text{and} \qquad \Phi_2(0;\lambda)=\binom{0}{1}.
  \end{equation}
Then it follows from  \eqref{syst.7} and \eqref{9} that in any of
the above half-planes
  \begin{eqnarray}\label{syst.8}
  \Phi_1(x;\lambda)=(1+o(1))Y_1(x;\lambda)+o(1)Y_2(x;\lambda),\nonumber  \\
  \Phi_2(x;\lambda) = o(1)Y_1(x;\lambda) + (1+o(1))Y_2(x;\lambda).
  \end{eqnarray}
Hence the corresponding characteristic determinant is
  \begin{multline}\label{syst.10}
  \Delta(\lambda)=\det
  \begin{pmatrix}
        1                        &  -h_0 \\
        \varphi_{11}(1;\lambda)  &  \varphi_{12}(1;\lambda)-h_1
  \end{pmatrix}
= \\
=-h_1+h_0\varphi_{11}(1;\lambda)+\varphi_{12}(1;\lambda)=-h_1+h_0
e^{i{b_1}\lambda }+o(e^{i{b_1}\lambda})+o(e^{i{b_2}\lambda}).
  \end{multline}
The vector function
  \begin{equation}\label{syst.11}
  w(x;\lambda)=\binom{w_1(x;\lambda)}{w_2(x;\lambda)}=h_0\Phi_1(x;\lambda)+\Phi_2(x;\lambda),
  \quad \lambda\in \Bbb C,
  \end{equation}
satisfies both the equation~\eqref{3.1} and the first of the
boundary conditions~\eqref{5.2}.
Let the vector function $f(x)=\col(f_1(x),f_2(x))$ be orthogonal to the system
of root functions of problem~\eqref{3.1}, \eqref{5.2}. Then the quotient
  \begin{equation}\label{syst.12}
  F(\lambda)=\frac{\left( w(x;\lambda),f(x) \right)}{\Delta(\lambda)}=
  \frac{\int_0^1
  \left(w_1(x;\lambda)\overline{f_1(x)}+w_2(x;\lambda)\overline{f_2(x)}\right)\,dx}
  {-h_1 + h_0e^{i{b_1}\lambda } + o(e^{i{b_1}\lambda})+o(e^{i{b_2}\lambda})}
  \end{equation}
is entire function of at most first growth.

Introduce the sector $S_{b_1,b_2}$ by setting
  \begin{equation}\label{syst.13}
S_{b_1,b_2} := \{\theta\in \Bbb C:\   0<\Re (i{b_2}\theta)<\Re
(i{b_1}\theta)\}.
  \end{equation}
Then, for $t\to +\infty$, we obtain:
  \begin{eqnarray}\label{syst.14}
  \int_0^1 \left(w_1(x;\theta t)\overline{f_1(x)}+w_2(x;\theta t)\overline{f_2(x)}\right)\,dx \nonumber  \\
  =O\left(\int_0^1 |e^{i{b_1}\theta
  tx}|(|f_1(x)|+|f_2(x)|)\,dx\right) = o(|e^{i{b_1}\theta t}|),\quad \theta\in  S_{b_1,b_2}.
  \end{eqnarray}
Similarly, we have
\begin{equation}\label{syst.15}
  \Delta(\lambda)= {\Delta(\theta t)} = -h_1+h_0e^{i{b_1}\theta t}+o(e^{i{b_1}\theta t})+o(e^{i{b_2}\theta t})\sim | h_0 e^{i{b_1}\theta t}|,
\quad  \theta\in  S_{b_1,b_2},
     \end{equation}
as $t\to \infty.$ Combining  \eqref{syst.14} with \eqref{syst.15}
we arrive at the relation
  \begin{equation}\label{syst.16}
\lim_{t\to +\infty}F(\theta t) = \lim_{t\to +\infty}
\frac{\int_0^1 \left(w_1(x;\theta t)\overline{f_1(x)} +
w_2(x;\theta t)\overline{f_2(x)}\right)\,dx} {\Delta(\theta t)}
=0,\qquad \theta\in  S_{b_1,b_2}.
  \end{equation}
On the other hand, for $\theta\in  S_{b_1,b_2}$ one gets
  \begin{eqnarray*}
  \int_0^1 \left(w_1(x;\theta t)\overline{f_1(x)}+w_2(x;\theta t)\overline{f_2(x)}\right)\,dx  \nonumber \\
  =O\left(\int_0^1 |e^{i{b_2}\theta tx}|(|f_1(x)|+|f_2(x)|)\,dx\right)\to
  0\quad \text{as}\quad  t\to -\infty,\quad  \theta\in
  S_{b_1,b_2},
  \end{eqnarray*}
 and
  \begin{equation*}
  \Delta(\lambda)= {\Delta(\theta t)} = -h_1 + h_0 e^{i{b_1}\theta t}+o(e^{i{b_1}\theta t})+o(e^{i{b_2}\theta t})\to -h_1
  \end{equation*}
as $t\to -\infty$, $\theta\in  S_{b_1,b_2}$.

Combining these estimates we obtain
  \begin{equation}\label{syst.19}
  \lim_{t\to -\infty}
  \frac{\int_0^1 \left(w_1(x;\theta t)\overline{f_1(x)}+w_2(x;\theta t)\overline{f_2(x)}\right)\,dx}
  {\Delta(\theta t)}
  =0, \qquad \theta\in  S_{b_1,b_2}.
      \end{equation}

Choose numbers $\theta_1, \ \theta_2\in S_{b_1,b_2}$ not lying on
the same line  with the origin. Then the rays $\theta_1t,\
\theta_2t\ (t>0)$ and $\theta_1t,\ \theta_2t\ (t<0)$ divide the
complex plane into four sectors with openings less than $\pi$. It
follows from  estimates \eqref{syst.16} and \eqref{syst.19} that
the function  $F(\cdot)$ is bounded on these rays. Being an entire
function of order not exceeding one, the function $F(\cdot)$ is
bounded on each of these sectors, by the Phragmen-Lindel\"of
theorem. Thus, $F(\cdot)$ is bounded on the whole complex plane
and, by the Liouville theorem, it is a constant. It follows from
\eqref{syst.19} that $F(\lambda)\equiv 0$.

Thus, the vector function $f(x)$ is orthogonal to $w(x;\lambda)$
for all $\lambda$. In particular, it is  orthogonal to  all
solutions of the system~\eqref{3.1} subject to the following
boundary conditions
  \begin{equation}\label{syst.20}
  \begin{cases}
y_1(0)=h_0 y_2(0) \\
y_1(1)=y_2(1).
  \end{cases}
  \end{equation}
In this case $J_{13}J_{24}\neq0$, and conditions \eqref{syst.20}
are weakly regular. By Theorem \ref{th2.1}, the system  of the
root functions  of the  problem~\eqref{3.1}, \eqref{syst.20} is
complete in $L^2\bigl([0,1];{\Bbb C}^2\bigr).$   Hence
$f(x)\equiv0$.

The minimality property is implied by Lemma~\ref{Lem_minimality}.
      \end{proof}
Theorems~\ref{th2.1} and~\ref{th7.1} make it possible to describe
all boundary conditions for systems~\eqref{3.1} with $Q=0$ such
that the root functions system of the problem~\eqref{3.1},
\eqref{3.3} is complete.
  \begin{corollary}\label{cor6.2}
Let $Q=0$ in the assumptions of Theorem~\ref{th7.1}. Then the
system of root functions of problem~\eqref{3.1}, \eqref{3.3} is
incomplete if and only if the pair of the boundary
conditions~\eqref{3.3} is equivalent to that contained  at least
one of the "Volterra"\ conditions: $y_j(0)=0$ or $y_j(1)=0, \ j\in
\{1,2\}$.
  \end{corollary}

\begin{proof}[Proof]
\emph{Necessity.} Assume for simplicity that one of the boundary
conditions is of the form $y_1(0)=0$. Then the system of root
functions of problem~\eqref{3.1}, \eqref{3.3} is either empty or
has the form $\{\col(0, e^{(2 \pi i (n+\alpha) x)})\}_{n \in
\mathbb{Z}}$ for some $\alpha \in \mathbb{C}.$ Clearly, it is
incomplete in  $L^2\bigl([0,1];{\Bbb C}^2\bigr)$.

\emph{Sufficiency.} Assume that the system of root functions is
incomplete. Then by Theorem~\ref{th2.1} condition~\eqref{5.1} is
violated. Without loss of generality we can assume that $J_{14} =
0$ and $J_{34} = 0$. Consider two cases.

\emph{(i)} $J_{13} = 0$. Then the matrix composed of $1^{st}$,
$3^{rd}$ and $4^{th}$ columns of the matrix $(C\,\,D)$ has rank 1.
By equivalent transformations the matrix $(C\,\,D)$ of boundary
conditions  is reduced to the matrix with the only one non-zero
entry in the second row. In other words, one of the boundary
conditions is reduced to a "Volterra"\ condition $y_2(0)=0$.

\emph{(ii)} $J_{13} \ne 0$. Then the boundary conditions are
equivalent to the following ones
$$
y_1(0)=h_0 y_2(0),\quad y_1(1)=h_1 y_2(0),
$$
that is, to conditions~\eqref{5.2} with arbitrary $h_0,h_1$.

By Theorem~\ref{th7.1} we have $h_0 h_1=0$. Hence again one of the
condition is of Volterra type.
\end{proof}
We emphasize that as distinct from Theorem \ref{th3.1} the
assumptions of Theorem \ref{th7.1}  do not depend on $Q$.
Moreover,  Theorem \ref{th7.1} shows that
Proposition~\ref{prop.neces} is no longer valid whenever $B\not =
B^*.$ In other words, as distinct from the case of $B=B^*,$
\emph{the weak regularity of boundary conditions \eqref{1.3} is
not necessary for completeness of the operator $L_{C,D}(0)$ with
$Q=0.$} However, the following criterion takes place.
   \begin{corollary}\label{cor6.4}
Let $n=2$ and  $B = \diag(b^{-1}_1,b^{-1}_2)$ with  $a:= b_1
b^{-1}_2\notin{\Bbb R}$. Then the boundary conditions \eqref{3.3}
are weakly regular if and only if both operators $L_{C,D}(0)$ and
$L_{C,D}(0)^*$ are complete in $L^2\bigl([0,1];{\Bbb C}^2\bigr)$.
   \end{corollary}
   \begin{proof}[Proof]
\emph{Necessity} is implied by Theorem \ref{th2.1} and Corollary
\ref{cor2.1A}.

\emph{Sufficiency.} Assume that both operators $L_{C,D}(0)$ and
$L_{C,D}(0)^*$ are complete  in $L^2\bigl([0,1];{\Bbb C}^2\bigr)$
but the BC \eqref{3.3} are not weakly regular. Then, by Corollary
\ref{cor6.2},  we can assume  that BC are equivalent to conditions
\eqref{5.2}. In this case the adjoint operator $L^*_{C,D}$ is
defined by the differential expression $L^* = -iB^*d/dx + Q^*(x)$
and the boundary conditions
     \begin{equation}
\overline{h_0} y_1(0)+\overline{a} y_2(0) - \overline{h_1}
y_1(1)=0, \qquad y_2(1) = 0.
     \end{equation}
The second condition is  of Volterra type and,  by
Corollary~\ref{cor6.2}, operator $L^*_{C,D}$ is incomplete. This
contradicts the assumption.
   \end{proof}

\noindent {\bf Acknowledgments.} We are grateful to Anton Lunyov
for careful reading the manuscript and making numerous remarks. We
are also appreciate his kind help with preparing
Examples~\ref{ex.k.k+1} and~\ref{ex.3.3}.

\bigskip

Institute of Applied Mathematics and Mechanics of NASU, Donetsk
\qquad

e-mail: mmm@telenet.dn.ua

Donetsk National University, Donetsk

e-mail: oridoroga@skif.net


\medskip

\begin{thebibliography}{99}

\bibitem{AgrKatSiVoi99}
M.S.~Agranovich, B.Z.~Katsenelenbaum, A.N.~Sivov and
N.N.~Voitovich, Generalized method of eigenoscilations in
dirfaction theory, WILEY-VCH Verlag Berlin GmbH, Berlin, (1999).

\bibitem{Bir08}
G.D.~Birkhoff, On the asymptotic character of the solution of the
certain linear differential equations containing parameter,
\emph{Trans. Amer. Math. Soc.} \textbf{9} (2) (1908), pp.
219--231.

\bibitem{Bir08exp}
G.D.~Birkhoff, Boundary value and expansion problems of ordinary
linear differential equations, \emph{Trans. Amer. Math. Soc.}
\textbf{9} (1908), pp. 373--395.

\bibitem{BirLan23}
G.D.~Birkhoff and R.E.~Langer, The boundary problems and
developments associated with a system of ordinary differential
equations of the first order, \emph{Proc. Am. Acad. Arts Sci.}
\textbf{58} (1923), pp. 49--128.

\bibitem{DjaMit06}
P.~Djakov and B.~Mityagin, Instability zones of one-dimensional
periodic Schr$\ddot{{\rm o}}$dinger and Dirac operators, (Russian)
\emph{Uspekhi Mat. Nauk} \textbf{61} (4) (2006), pp 77--182;
translation in \emph{Russian Math. Surveys} \textbf{61} (4)
(2006), pp. 663--766.

\bibitem{DjaMit09Sing}
P.~Djakov and B.~Mityagin, Bari-Markus property for Riesz
projections of Hill operators with singular potentials,
\emph{Contemp. Math.} \textbf{481} (2009), pp. 59--80.

\bibitem{DjaMit09}
P.~Djakov and B.~Mityagin, Bari-Markus property for Riesz
projections of 1D periodic Dirac operators, \emph{Math. Nachr.}
\textbf{283} (3) (2010), pp. 443--462.

\bibitem{DjaMit09Hill}
P. Djakov and B. Mityagin, Convergence of spectral decompositions
of Hill operators with trigonometric polynomial potentials,
\emph{arXiv} 0911.3218 (Submitted on 17 Nov 2009).

\bibitem{DjaMit10Trig}
P. Djakov and B. Mityagin, 1D Dirac operators with special
periodic potentials, \emph{arXiv} 1007.3234 (Submitted on 19 Jul
2010).

\bibitem{DjaMit10}
P. Djakov and B. Mityagin, Unconditional convergence of spectral
decompositions of 1D Dirac operators with regular boundary
conditions, \emph{arXiv} 1008.4095 (Submitted on 24 Aug 2010).
Accepted in IUMJ.

\bibitem{DjaMit11Crit}
P. Djakov and B. Mityagin, Criteria for existence of Riesz bases
consisting of root functions of Hill and 1D Dirac operators,
\emph{arXiv} 1106.5774 (Submitted on 28 Jun 2011).

\bibitem{DjaMit11Equi}
P. Djakov and B. Mityagin, Equiconvergence of spectral
decompositions of 1D Dirac operators with regular boundary
conditions, \emph{arXiv} 1108.0344 (Submitted on 1 Aug 2011).

\bibitem{DjaMit11CritDir}
P. Djakov and B. Mityagin, Riesz bases consisting of roo functions
of 1D Dirac operators, \emph{arXiv} 1108.4225 (Submitted on 22 Aug
2011).

\bibitem{DunSch71}
N. Dunford, J. Schwartz, Linear Operators, Part III, Spectral
Operators, Wiley, New York, 1971.

\bibitem{GesTka09}
F.~Gesztesy and V.~Tkachenko, A criterion for Hill operators to be
spectral operators of scalar type, \emph{J. Analyse Math.}
\textbf{107} (2009), pp. 287--353.

\bibitem{GesTka11}
F.~Gesztesy and V.~Tkachenko, A Schauder and Riesz basis criterion
for non-selfadjoint Schr\"{o}dinger operators with periodic and
anti-periodic boundary conditions, \emph{arXiv} 1104.4846 (Apr 26
and June 4, 2011)

\bibitem{Gin71}
Yu.P.~Ginzburg, The almost invariant spectral propeties of
contractions and the multiplicative properties of analytic
operator-functions, (Russian) \emph{Funktsional. Anal. i
Prilozen.} \textbf{5} (3) (1971), pp. 32--41; translation in
\emph{Funct. Anal. Appl.} \textbf{5} (3) (1971), pp. 197--205.

\bibitem{GohKre65}
I.C.~Gohberg and M.G.~Krein, Introduction to the theory of linear
nonselfadjoint operators in Hilbert space, Nauka, Moscow (1965);
English translation: \emph{Translations of Mathematical
Monographs}, \textbf{vol. 18}, American Mathematical Society,
Providence, R.I. (1969).

\bibitem{HasOri09}
S. Hassi and L. Oridoroga, Theorem of Completeness for a
Dirac-Type Operator with Generalized $\lambda$-Depending Boundary
Conditions, \emph{Integral Equat. Oper. Theor.} \textbf{64}
(2009), pp. 357--379.

\bibitem{Kel51}
M.V.~Keldysh, On the characteristic values and characteristic
functions of certain classes of non-self-adjoint equations,
(Russian) \emph{Doklady Akad. Nauk SSSR (N.S.)} \textbf{77} (1)
(1951), pp. 11--14.

\bibitem{Kes64}
G. M. Keselman, On the unconditional convergence of eigenfunction
expansions of certain differential operators, \emph{Izv. Vyssh.
Uchebn. Zaved. Mat.} \textbf{39} (2) (1964), pp. 82--93 (Russian).

\bibitem{Khr3}
A.P.~Khromov, Finite dimensional perturbations of Volterra
operators, (Russian) \emph{Sovrem. Mat. Fundam. Napravl.}
\textbf{10} (2004), pp. 3--163; translation in \emph{J. Math.
Sci.} (N. Y.) \textbf{138} (5) (2006), pp. 5893--6066.

\bibitem{KosShk78}
A.G.~Kostyuchenko and A.A.~Shkalikov, Summability of expansions in
eigenfunctions of differential operators and of convolution
operators, (Russian) \emph{Funktsional. Anal. i Prilozhen.}
\textbf{12} (4) (1978), pp. 24--40; translation in \emph{Funct.
Anal. Appl.} \textbf{12} (4) (1978), pp. 262--276.

\bibitem{Lev56}
B.Ya.~Levin,  Distribution of zeros of entire functions,
\emph{Transl. Math. Monographs} \textbf{vol.5}, American
Mathematical Society, Providence, R.I. (1964).

\bibitem{LevSar88}
B.M.~Levitan  and I.S.~Sargsyan, Sturm-Liouville And Dirac
Operators, Nauka, Moscow (1988); English translation: Kluzer,
Dordrechz (1991).

\bibitem{LunMal2011}
A.A. Lunyov and M.M. Malamud, On completeness of root vectors
systems of boundary value problems  for first order systems with
non-weakly regular boundary conditions (in preparation).

\bibitem{Mal99}
M.M.~Malamud, Questions of uniqueness in inverse problems for
systems of differential equations on a finite interval,
\emph{Trans. Moscow Math. Soc.} \textbf{60} 1999,  pp. 173--224.

\bibitem{Mal08}
M.M.~Malamud, On the completeness of a system of root vectors of
the Sturm-Liouville operator with general boundary conditions,
\emph{Funct. Anal. Appl.} \textbf{42} (3) (2008), pp. 198--204.

\bibitem{MalOri00}
M.M.~Malamud and L.L.~Oridoroga, Completeness theorems for systems
of differential equations, \emph{Funct. Anal. Appl.} \textbf{34}
(4) (2000), pp. 308--310.

\bibitem{MalOri10}
M.~M.~Malamud and L.~L.~Oridoroga, On the Completeness of the Root
Vectors of First Order Syst1ems, \emph{Doklady Mathematics},
\textbf{82} (3) (2010), pp. 899--905.

\bibitem{Mar77}
V.A.~Marchenko, Sturm-Liouville operators and applications.
Naukova dumka, Kiev (1977); English translation: \emph{Operator
Theory: Advances and Applications} \textbf{vol. 22}, Birkhauser
Verlag, Basel (1986).

\bibitem{Markus86}
A.S. Markus,  {An Introduction to the Spectral Theory of
Polynomial Operator Pencils} (Shtitsa, Chisinau, 1986) [in
Russian].

\bibitem{MasVor70}
V.P.~Maslov and G.A.~Voropaeva, Multiple completeness in the sense
of M.V.~Keldysh and uniqueness of the solution of the
corresponding Cauchy problem, (Russian) \emph{Funktsional. Anal. i
Prilozen.} \textbf{4} (2) (1970), pp. 10--17; translation in
\emph{Funct. Anal. Appl.} \textbf{4} (2) (1970), pp. 99--105.

\bibitem{Mikh62}
V. P. Mikhailov, On Riesz bases in $L^2(0, 1)$, \emph{Dokl. Akad.
Nauk SSSR} \textbf{144} (1962), pp. 981--984 (Russian).

\bibitem{Mit04}
B.~Mityagin, Spectral expansions of one-dimensional periodic Dirac
operators, \emph{Dyn. Partial Differ. Equ.} \textbf{1} (2004), pp.
125--191.

\bibitem{Nai69}
M.A.~Naimark, Linear differential operators, Nauka, Moscow (1969);
English translation: Part I: Elementary theory of linear
differential operators, Frederick Ungar Publishing Co., New York
(1967); Part II: Linear differential operators in Hilbert space
(1968).

\bibitem{ZMNovPit80}
S.P.~Novikov, S.V.~Manakov, L.P.~Pitaevskij, V.E.~Zakharov, Theory
of solitons. The inverse scattering method. Transl. from the
Russian., Contemporary Soviet Mathematics. New York - London:
Plenum Publishing Corporation. Consultants Bureau. (1984).

\bibitem{Shk76}
A.A.~Shkalikov, The completeness of the eigen- and associated
functions of an ordinary differential operator with nonregular
splitting boundary conditions, \emph{Funct. Anal. Appl.}
\textbf{10} (4) (1976), pp. 305--316.

\bibitem{Shk79}
A. A. Shkalikov, On the basis problem of the eigenfunctions of an
ordinary differential operator, \emph{Russ. Math. Surv.}
\textbf{34} (5) (1979), pp. 249--250.

\bibitem{Shk82}
A. A. Shkalikov, The basis problem of the eigenfunctions of
ordinary differential operators with integral boundary conditions,
\emph{Moscow Univ. Math. Bull.} \textbf{37} (6) (1982), pp.
10--20.

\bibitem{Tam12}
J. D. Tamarkin, Sur quelques points de la theorie des equations
differentielles lineaires ordinaires et sur la generalisation de
la serie de Fourier, \emph{Rend. Circ. Mat. Palermo} \textbf{34}
(2) (1912), pp. 345--382.

\bibitem{Tam17}
J. D. Tamarkin, On some general problems of the theory of ordinary
linear differential operators and on expansion of arbitrary
function into serii, Petrograd. 1917, 308 p.

\bibitem{Tam28}
J. D. Tamarkin, Some general problems of the theory of linear
differential equations and expansions of an arbitrary functions in
series of fundamental functions, \emph{Math. Z.} \textbf{27}
(1928), pp. 1--54.

\bibitem{TroYam01}
I. Trooshin and M. Yamamoto, Riesz basis of root vectors of a
nonsymmetric system of first-order ordinary differential operators
and application to inverse eigenvalue problems, \emph{Appl. Anal.}
\textbf{80} (2001), pp. 19--51.

\bibitem{TroYam02}
I. Trooshin and M. Yamamoto, Spectral properties and an inverse
eigenvalue problem for nonsymmetric systems of ordinary
differential operators, \emph{J. Inverse Ill-Posed Probl.}
\textbf{10} (6) (2002), pp. 643--658.

\end{thebibliography}
\end{document}